\pdfoutput=1
\documentclass[a4paper,12pt]{amsart}
\usepackage[utf8]{inputenc}

\usepackage{amsfonts,amsmath,amssymb,amsthm}
\DeclareMathAlphabet{\mathbbold}{U}{bbold}{m}{n}
\usepackage{bbm}
\usepackage{xspace}
\usepackage{csquotes}
\usepackage{standalone}
\usepackage{caption}
\usepackage{subcaption} 
\usepackage[spacing,kerning,tracking]{microtype}
\microtypecontext{spacing=nonfrench}
\usepackage[x11names,table]{xcolor}
\usepackage[naturalnames=false]{hyperref}
\usepackage{graphicx}
\graphicspath{ {img/} }
\usepackage{tikz}
\usetikzlibrary{calc}
\usetikzlibrary{cd}
\usetikzlibrary{3d}
\usepackage{csquotes}
\usepackage{booktabs}
\hypersetup{
	colorlinks = true,          
	urlcolor = DeepSkyBlue3, linkcolor = Chartreuse4, citecolor = DarkOrange2
}
\usepackage{fullpage}
\RequirePackage[backend=bibtex,isbn=false,url=false,maxbibnames=5,bibencoding=utf8]{biblatex}
\renewbibmacro{in:}{, }
\DeclareFieldFormat[misc]{title}{\mkbibquote{#1}}
\DeclareFieldFormat[article]{volume}{\mkbibbold{#1}}
\DeclareFieldFormat{url}{\textsc{url}: \href{#1}{#1}}
\DeclareFieldFormat{doi}{\textsc{doi}: \href{http://dx.doi.org/#1}{#1}}
\DeclareFieldFormat{eprint:arxiv}{arXiv:\href{https://arxiv.org/abs/#1}{#1}}
\newcommand\polymake{\texttt{polymake}\xspace}

\newcommand\mptopcom{\texttt{mptopcom}\xspace}
\newcommand\TOPCOM{\texttt{TOPCOM}\xspace}

\newcommand\topcom\TOPCOM
\newcommand\ZZ{\mathbb{Z}}
\newcommand\RR{\mathbb{R}}

\newcommand\1{{\mathbbold 1}}

\DeclareMathOperator{\seccone}{sec}
\DeclareMathOperator{\MC}{MC}
\DeclareMathOperator{\MF}{MF}
\DeclareMathOperator{\Dr}{Dr}
\DeclareMathOperator{\TGr}{TGr}
\DeclareMathOperator{\Sym}{Sym}
\DeclareMathOperator{\conv}{conv}
\DeclareMathOperator{\vol}{vol}

\DeclareMathOperator{\Vertices}{Vert}
\newcommand\bigmid{\,\bigl\vert\bigr.\,}
\newcommand{\DualGraph}{G^{\vee}}
\newcommand{\PC}{\mathcal{C}}
\newcommand{\envelope}{\mathcal{E}}

\theoremstyle{theorem}
\newtheorem{theorem}{Theorem}
\newtheorem{proposition}[theorem]{Proposition}
\newtheorem{lemma}[theorem]{Lemma}
\newtheorem{corollary}[theorem]{Corollary}

\theoremstyle{definition}
\newtheorem{example}[theorem]{Example}
\theoremstyle{remark}
\newtheorem{remark}[theorem]{Remark}
\newtheorem{observation}[theorem]{Observation}
\newtheorem{question}[theorem]{Question}

\title[Hypersimplices with a view]{Subdivisions of Hypersimplices: \\ with a View Toward Finite Metric Spaces}

\author{Laura Casabella}
\author{Michael Joswig}
\author{Lars Kastner}

\address[Laura Casabella]{
	Max-Planck Institute for Mathematics in the Sciences, Leipzig,
	\url{laura.casabella@mis.mpg.de}
}
\address[Michael Joswig]{
	Technische Universität Berlin,
	Chair of Discrete Mathematics/Geometry \\
	Max-Planck Institute for Mathematics in the Sciences, Leipzig,
	\url{joswig@math.tu-berlin.de}
}
\address[Lars Kastner]{
	Technische Universität Berlin,
	Chair of Discrete Mathematics/Geometry,
	\url{lkastner@math.tu-berlin.de}
}

\thanks{MJ has been supported by the Deutsche Forschungsgemeinschaft (DFG, German Research Foundation) under Germany's Excellence Strategy -- The Berlin Mathematics Research Center MATH$^+$ (EXC-2046/1, project ID 390685689) and \enquote{Symbolic Tools in Mathematics and their Application} (TRR 195, project ID 286237555).
LK has been supported by the MaRDI project \cite{mardi} under DFG project ID 460135501.
}

\begin{document}

\begin{abstract}
  The secondary fan $\Sigma(k,n)$ is a polyhedral fan which stratifies the regular subdivisions of the hypersimplices $\Delta(k,n)$.
  We find new infinite families of rays of $\Sigma(k,n)$, and we compute the fans $\Sigma(2,7)$ and $\Sigma(3,6)$.
  In the special case $k=2$ the fan $\Sigma(2,n)$ is closely related to the metric fan $\MF(n)$, which forms a natural parameter space for the metric spaces on $n$ points.
  So our results yield a classification of the finite metric spaces on seven points.
\end{abstract}

\maketitle

\section{Introduction}

The hypersimplex $\Delta(k,n)$ is the convex hull of those $0/1$-vectors of length $n$ which have exactly $k$ ones.
Equivalently, this is the matroid (base) polytope of the uniform matroid of rank $k$ on $n$ elements.
We are interested in the regular subdivisions of $\Delta(k,n)$; the latter are those polyhedral subdivisions which are induced by a height function.
The secondary fan, which we denote as $\Sigma(k,n)$, is a polyhedral fan which subdivides the space of height functions into cones which yield the same subdivision.
This line of research is motivated by tropical geometry.
The \emph{Dressian} $\Dr(k,n)$ is the subfan of $\Sigma(k,n)$ which comprises those lifting functions which induce a regular subdivision into matroidal cells.
In the language of tropical geometry, this means that the Dressian $\Dr(k,n)$ is the moduli space of uniform tropical linear spaces with parameters $k$ and $n$.
Our results contribute to exploring the embedding of $\Dr(k,n)$ into the secondary fan.
For tropical geometry, we refer to the textbooks \cite{Tropical+Book} and \cite{ETC}.

Another reason to study the regular subdivisions of hypersimplices comes from finite metric spaces, which are concerned with the special case $k=2$.
Sturmfels and Yu \cite{SturmfelsYu:2004} saw that each metric space on $n$ points has a natural interpretation as a height function on the second hypersimplex $\Delta(2,n)$.
Conversely, up to some technicalities, such a height function gives rise to a metric.
In the context of phylogenetics, Bandelt and Dress \cite{BandeltDress:1992} obtained a fundamental decomposition theorem for finite metric spaces on $n$ points; see also \cite[\S11.1.2]{DezaLaurent:2010}.
In polyhedral geometry terms, their split decomposition theorem \cite[Theorem~2]{BandeltDress:1992} amounts to characterizing the rays of the Dressian $\Dr(2,n)$ as a subfan of $\Sigma(2,n)$.
Following this line of research, it is tempting to ask for all the rays of $\Sigma(2,n)$, or even $\Sigma(k,n)$ for arbitrary $k$ and $n$.
However, computational evidence suggests that this is hopeless in general, because there are way too many; see Section~\ref{sec:computation} below.
The split decomposition theorem of Bandelt and Dress has been generalized to arbitrary polytopes by Hirai \cite{Hirai:2006a} and Herrmann and Joswig \Cite{HerrmannJoswig:2008}.

The reasoning above invites the search for at least some of the rays of $\Sigma(k,n)$.
This is potentially useful for tropical geometry because this could shed some light on the unresolved question of how exactly the tropical Grassmannian $\TGr(k,n)$ sits inside $\Dr(k,n)$.
While the tropical Grassmannian is contained in the Dressian as a set, it inherits its fan structure from the Gröbner fan of the Plücker ideal \cite[\S4.3]{Tropical+Book}.
This is known \emph{not} to be compatible with the secondary fan structure in general.
By analyzing rays of $\Sigma(k,n)$ which are somehow near $\Dr(k,n)$, we find regular subdivisions of $\Delta(k,n)$ with rich connections to algebraic combinatorics.
The coarsest matroid subdivisions of $\Delta(d,n)$ for $d\leq 4 $ and $n\leq 10$ have been determined by Schr\"oter \cite{Schroeter:2019}.
In the phylogenetic scenario, finding new rays of $\Sigma(2,n)$ leads to interesting generalizations of the split decomposition theorem.

Our main results are the following.
First, for $2\leq k \leq n-2$, we construct two distinct non-split coarsest regular subdivisions of $\Delta(k,n)$ (Proposition~\ref{thm:lambda}, Theorem~\ref{thm:lambda}, Theorem~\ref{thm:kappa}).
In the special case $k=2$ these yield metrics on $n$ points which are split-prime.
To the best of our knowledge all these subdivisions are new, except for the cases $(k,n)\in\{(2,4),(2,5),(2,6)\}$ where all the coarsest subdivisions of $\Delta(k,n)$ had been classified by Koolen, Moulton and T\"onges \cite{KoolenMoultonTonges:2000} and Sturmfels and Yu \cite{SturmfelsYu:2004}.
None of our new rays lie on the respective Dressians $\Dr(k,n)$.
Hence they do not occur in Schr\"oter's classification \cite{Schroeter:2019}, and none of the corresponding coarsest subdivisions is a multi-split in the sense of Herrmann~\cite{MR2739494}.
Second, we obtain full descriptions of the secondary fans of $\Delta(2,7)$ (Theorem~\ref{thm:computation:Delta27}) and $\Delta(3,6)$ (Theorem~\ref{thm:computation:Delta36}).
The secondary fan of $\Delta(2,7)$, or rather its close relative, the metric fan $\MF(7)$, classifies the metric spaces on seven points; this leads to a direct generalization of the main results of \cite{SturmfelsYu:2004}.

The paper is structured as follows.
We start out in Section \ref{sec:subdivisions} by recalling the basics on polyhedral subdivisions of arbitrary point configurations.
On the way we collect a new sufficient condition for a subdivision to be coarsest.
Section~\ref{sec:hypersimplices} introduces hypersimplices, with a sketch of their role in tropical geometry.
Our constructions of new coarsest regular subdivisions of hypersimplices are presented in Section~\ref{sec:eulerian}.
Afterwards, in Section~\ref{sec:metric} we study the connection between subdivisions of $\Delta(2,n)$ and metric spaces on $n$ points.
Our computational results are summarized in Section~\ref{sec:computation}, and the final Section~\ref{sec:conclusion} comprises two open questions.



\section{Regular Subdivisions and Their Tight Spans}
\label{sec:subdivisions}
We start out with recalling basic concepts from polyhedral geometry, to fix our notation.
To the reader we recommend the textbook \cite{Triangulations} as a general reference on this subject.
Let $A\subset\RR^d$ be a finite set of points.
To simplify our setup we will assume that those points lie in \emph{convex position}, i.e., $A$ agrees with the set of vertices of its convex hull $\conv(A)$, which is a polytope.
Furthermore, we require that $A$ affinely spans the entire space, i.e., we have $\dim\conv(A)=d$.
A finite set of polytopes, $S$, (and all their faces) form a \emph{polyhedral subdivision} of $A$ if the polytopes in $S$ meet face to face and cover $\conv(A)$.
These polytopes are the \emph{cells} of $S$, and they are partially ordered by inclusion.
The \emph{spread} of $S$ is the number of its maximal cells.

Let $\omega:A\to\RR$ be an arbitrary function, which we call a \emph{lifting}.
Then projecting the lower faces of the \emph{lifted polytope} $\conv\{(a,\omega(a)) \mid a\in A\} \subset \RR^{d+1}$ by omitting the last coordinate induces a subdivision of $A$.
Such a subdivision is called \emph{regular}, and we denote it as $A^\omega$.
The \emph{lower faces} of the lifted polytope are those which admit a downward pointing outer normal vector.

We associate another polyhedron with the pair $(A,\omega)$.
The \emph{(lower) envelope} of $A$ with respect to the lifting function $\omega$ is defined as
\begin{equation}\label{eq:envelope}
  \envelope_{\omega}(A) \ := \ \bigl\{x\in\RR^{d+1} \bigmid A'x \geq -\omega \bigr\} \enspace, 
\end{equation}
where $A'$ is the homogenization of $A$, i.e., $A'$ is a matrix where each row corresponds to one point in $A$, with an additional coefficient one appended.
The envelope is always unbounded upward, i.e., in the positive $(d{+}1)$st coordinate direction.
The \emph{tight span} of a subdivision $S$ of $A$ is the dual cell complex of $S$. 
If $S$ is regular and induced by a lifting $\omega$, then the tight span carries a natural structure as the polyhedral complex of the bounded faces of $\envelope_{\omega}(A)$ (see \cite[Theorem 10.59]{ETC}).
Note that the tight span is always connected.
In particular, its vertex set $\Vertices{\envelope_{\omega}(A)}$ is never empty.

\begin{remark}\label{rem:upper-lower}
  It also makes sense to look at the upper faces, instead of the lower ones, to define regular subdivisions.
  By preferring lower convex hulls we follow the convention of the textbook \cite{Triangulations} as well as the software systems \TOPCOM \cite{TOPCOM-paper} and \mptopcom \cite{mptopcom-paper}.
  Upper convex hulls may be simulated by passing to the negative height function.
\end{remark}

The regular subdivisions of $A$ are partially ordered by refinement \cite[\S2.2.3]{Triangulations}.
The minimal elements in this poset are the regular triangulations, while the maximal ones are the \emph{coarsest} regular subdivisions.
A subdivision of $A$ with exactly two maximal cells is a \emph{split} of $A$; this is coarsest, and it is always regular \cite[Lemma 3.5]{HerrmannJoswig:2008}.
By definition, the splits are precisely the (coarsest) subdivisions of spread two.
For a (regular) subdivision $S$ of $A$, the set $\{\omega:A\to\RR \mid A^\omega=S\}$ of lifting functions inducing the same subdivision $S$ is a relatively open polyhedral cone, $\seccone(\omega)$.
Its closure is the \emph{secondary cone}, which we denote as $\seccone(S)$.
The \emph{secondary fan}, $\Sigma(A)$ is the polyhedral fan which is comprised of all secondary cones.
Its lineality space has dimension $d+1$, and it is polytopal, i.e., it arises as the normal fan of some polytope.
The face lattice of $\Sigma(A)$ is anti-isomorphic to the refinement poset of the regular subdivisions of $A$; see \cite[Theorem 5.2.11]{Triangulations}.
In particular, the coarsest subdivisions correspond to the rays of the secondary fan.

Our present work is motivated by two known results.
The first one is a generalization of a key structural result for finite metric spaces, obtained by Bandelt and Dress \cite{BandeltDress:1992}; see Section~\ref{sec:metric} below for further explanations.
\begin{theorem}[{Hirai \cite{Hirai:2006a}; Herrmann and Joswig \cite[Theorem 3.10]{HerrmannJoswig:2008}}]
  \label{thm:split-decomposition}
  Let $\omega$ be a lifting function of a point set $A$.
  Then there is a coherent decomposition
  \begin{equation}
    \label{eq:split-decomposition}
    \omega \ = \ \omega_0+ \sum_{S \text{ split of } A} \lambda_S \omega_S \enspace ,
  \end{equation}
  where $\omega_0$ is split prime, and this decomposition is unique among all coherent decompositions of $\omega$.
\end{theorem}

The decomposition $\omega=\alpha+\beta$ of lifting functions is \emph{coherent} if $A^{\omega}$ is the common refinement of $A^{\alpha}$ and $A^{\beta}$. 
A lifting function $\omega$ is \emph{split prime} if none of the rays of the secondary cone of $A^{\omega}$ is a split. 
The lifting functions $\omega_{S}$ for the splits $S$ of $A$ can be chosen arbitrarily.
In the statement of the Theorem, the coefficient $\lambda_S$ denotes the coherency index of $\omega$ with respect to $\omega_{S}$, where $S$ is a split of $A$.

The \emph{coherency index} of $\omega$ with respect to some other height function, $\omega'$, is defined as
\begin{equation}\label{eq:coherency-index}
  \alpha^\omega_{\omega'} \ := 
  \min_{x\in\Vertices{\envelope_{\omega}(A)}}
  \left\{ \max_{x'\in\Vertices{\envelope_{\omega'}(A)}}
    \left\{ \min_{v\in V_{\omega'}(x')}
	\left\{ 
          \frac{\langle v, x \rangle +\omega(v)}{\langle v, x' \rangle +\omega'(v)}
	\right\}
      \right\}
    \right\}  \, ,
\end{equation}
where $V_{\omega'}(x')$ is the set of homogenized points $A'\subset\RR^{d+1}$ that are not contained in the cell dual to $x'$; see \cite[Eq.~(2)]{HerrmannJoswig:2008}.
This generalizes \cite[Eq.~(2)]{KoolenMoultonTonges:1998}.
Since $x'$ is a vertex of the envelope $\envelope_{\omega'}(A)$, that dual cell is a subpolytope of $\conv(A)$.
The coherency index measures how much a decomposition deviates from coherence.
Indeed, $\alpha_{\omega'}^{\omega}$ is the supremum of all positive $\lambda$ in  $\RR$ such that $(\omega - \lambda\omega', \lambda\omega')$ is a coherent decomposition of $\omega$.
For instance, we have $\alpha^\omega_{\omega}=1$ for any lifting function $\omega$.

Our first contribution (Corollary~\ref{cor:coarsest}) is an observation concerning general coarsest subdivisions.
Notice that none of the results in the remainder of this section requires the subdivisions in question to be regular.
The \emph{dual graph} $\DualGraph(\PC)$ of a pure polyhedral complex $\PC$ is the abstract graph whose nodes are the maximal cells.
There is an edge between two nodes for each codimension one face shared by the corresponding maximal cells.
The following observation is elementary and well-known, but we did not find a suitable reference; thus we include a short proof for the sake of completeness.
\begin{lemma}\label{lemma:dg:simple}
  Let $S$ be a subdivision of a polytope $P$.
  Then the dual graph $\DualGraph(S)$ is simple and connected, i.e., $\DualGraph(S)$ is a connected graph without loops or multiple edges.
\end{lemma}
\begin{proof}
  Assume that $\DualGraph(S)$ contains two nodes which are connected by at least two edges.
  That is to say, there are two maximal cells, $F$ and $G$, in $S$, such that
  \[
    F \cap G \ \supseteq \rho \cup \sigma,\ \mbox{ with } \rho \not= \sigma,
  \]
  where $\rho$ and $\sigma$ are both facets of the polytopes $F$ and $G$.
  This is a contradiction, as any two cells of $S$ meet in a common face, which may be empty.
  
  To show connectedness, pick two distinct maximal cells of $S$; call them $F$ and $G$ again.
  Let $f$ be a generic point in the interior of $F$, and let $g$ be a generic point in the interior of $G$.
  The line segment $L = [f,g]$ is contained in $P$, since $P$ is convex.
  Thus $L$ is covered by $S$.
  The generic choices of $f$ and $g$ entail that $L$ does meet any cell of $S$ of codimension greater than one.
  Following the cells that cover $L$ one gets the desired path from $F$ to $G$ in the dual graph of $S$.
\end{proof}
The following result will turn out to be useful for detecting coarsest subdivisions.
\begin{proposition}\label{prop:dg:contraction}
  Let $S$ and $S'$ be subdivisions of a polytope $P$.
  If $S'$ is a coarsening of $S$, then the dual graph $\DualGraph(S')$ arises from $\DualGraph(S)$ by contracting edges.
\end{proposition}
\begin{proof}
  Let $F$ and $G$ be two distinct maximal cells in $S$ which are adjacent, i.e., they span an edge in $\DualGraph(S)$.
  Then there are uniquely determined maximal cells $F'$ and $G'$ of $S'$ such that $F'$ contains $F$ and $G'$ contains $G$.

  Then two cases may occur.
  Either $F'$ and $G'$ are distinct.
  Then $F' \cap G' \supseteq F\cap G$, which has codimension one, and $\{F',G'\}$ is an edge of $\DualGraph(S')$.
  Or $F'=G'$, and the edge $\{F,G\}$ of $\DualGraph(S)$ gets contracted in $\DualGraph(S')$.

  Each edge of $\DualGraph(S')$ arises from the first case above, because each codimension one cell of $S'$ is a union of codimension one cells of $S$.
\end{proof}

\begin{corollary}\label{cor:coarsest}
  Let $A$ be a set of points in convex position, and let $S$ be a proper subdivision of $A$.
  If the dual graph of $S$ is complete, then $S$ is a coarsest subdivision.
\end{corollary}
\begin{proof}
  Without loss of generality, we may assume that the subdivision $S$ is not a split.
  Take $S'$ to be any coarsening of $S$.
  Then $\DualGraph(S')$ is obtained from $\DualGraph(S)$ by contraction according to Proposition~\ref{prop:dg:contraction}.
  Since $\DualGraph(S)$ is complete, any contraction will produce multiedges.
  But $\DualGraph(S')$ cannot have multiedges according to Lemma~\ref{lemma:dg:simple}.
  The only solution is that $\DualGraph(S')$ only consists of a single node, meaning that $S'$ is the trivial subdivision.
\end{proof}

A subdivision of $A$ is an \emph{$\ell$-split} if its tight span is a simplex of dimension $\ell-1$.
Here the splits reappear as $2$-splits.
Herrmann showed that $\ell$-splits are necessarily regular and coarsest \cite{MR2739494}.
In this way Corollary~\ref{cor:coarsest} is a generalization.
We also use the term \emph{multi-split} if the parameter $\ell$ is not important.

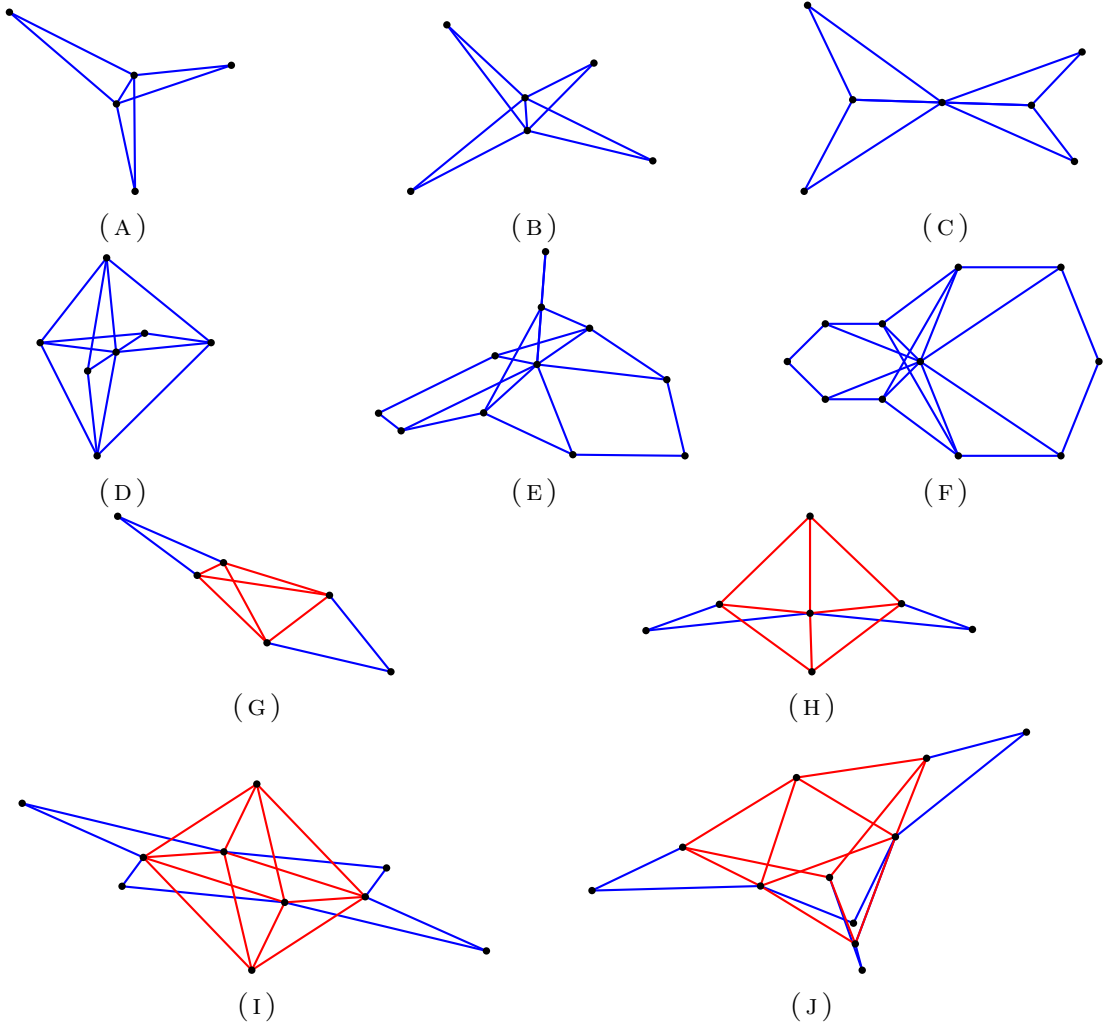
\begin{figure}
	\centering
	\begin{subfigure}[b]{0.32\textwidth}
		\centering

\begin{tikzpicture}[x  = {(1cm,0cm)},
                    y  = {(0cm,1cm)},
                    z  = {(0cm,0cm)},
                    scale = 1,
                    color = {lightgray}]

  \coordinate (v0_unnamed__1) at (1.39576, 0.3225, 1.15478);
  \coordinate (v1_unnamed__1) at (-0.116135, -0.189765, 0.193677);
  \coordinate (v2_unnamed__1) at (0.116135, 0.189765, -0.193677);
  \coordinate (v3_unnamed__1) at (-1.52457, 1.02628, 0.0901583);
  \coordinate (v4_unnamed__1) at (0.128803, -1.34878, -1.24494);

  \definecolor{vertexcolor_unnamed__1}{rgb}{ 0 0 0 }

  \tikzstyle{vertexstyle_unnamed__1} = [circle, scale=0.25pt, fill=vertexcolor_unnamed__1,]

  \definecolor{edgecolor_unnamed__1}{rgb}{ 0 0 0 }
  \tikzstyle{edgestyle_unnamed__1} = [thick,color=blue]


  \foreach \i/\k in {1/0,2/0,2/1,3/1,3/2,4/1,4/2} {
   \draw[edgestyle_unnamed__1] (v\i_unnamed__1) -- (v\k_unnamed__1);
  }

  \foreach \i in {0,1,3,2,4} {
    \node at (v\i_unnamed__1) [vertexstyle_unnamed__1] {};
  }

\end{tikzpicture}
		\caption{}
	\end{subfigure}
	\hfill
	\begin{subfigure}[b]{0.32\textwidth}
		\centering

\begin{tikzpicture}[x  = {(1cm,0cm)},
                    y  = {(0cm,1cm)},
                    z  = {(0cm,0cm)},
                    scale = 1.1,
                    color = {lightgray}]

  \coordinate (v0_unnamed__1) at (-0.0130873, 0.197271, 0.0788436);
  \coordinate (v1_unnamed__1) at (0.811345, 0.61588, 1.85744);
  \coordinate (v2_unnamed__1) at (-0.947673, 1.07676, -0.756232);
  \coordinate (v3_unnamed__1) at (1.5154, -0.562511, -1.29171);
  \coordinate (v4_unnamed__1) at (-1.38023, -0.92983, 0.189859);
  \coordinate (v5_unnamed__1) at (0.0142499, -0.197578, -0.0782049);

  \definecolor{vertexcolor_unnamed__1}{rgb}{ 0 0 0 }

  \tikzstyle{vertexstyle_unnamed__1} = [circle, scale=0.25pt, fill=vertexcolor_unnamed__1,]

  \definecolor{edgecolor_unnamed__1}{rgb}{ 0 0 0 }
  \tikzstyle{edgestyle_unnamed__1} = [thick,color=blue]


  \foreach \i/\k in {1/0,2/0,3/0,4/0,5/0,5/1,5/2,5/3,5/4} {
   \draw[edgestyle_unnamed__1] (v\i_unnamed__1) -- (v\k_unnamed__1);
  }

  \foreach \i in {1,4,0,5,2,3} {
    \node at (v\i_unnamed__1) [vertexstyle_unnamed__1] {};
  }

\end{tikzpicture}
		\caption{}
	\end{subfigure}
	\hfill
	\begin{subfigure}[b]{0.32\textwidth}
		\centering

\begin{tikzpicture}[x  = {(1cm,0cm)},
                    y  = {(0cm,1cm)},
                    z  = {(0cm,0cm)},
                    scale = 0.9,
                    color = {lightgray}]

    \coordinate (v0_unnamed__1) at (-1.96676, 1.43295, 0.764997);
  \coordinate (v1_unnamed__1) at (-1.30633, 0.0407313, -0.0291552);
  \coordinate (v2_unnamed__1) at (0, 0, 0);
  \coordinate (v3_unnamed__1) at (-2.01612, -1.30877, -0.853889);
  \coordinate (v4_unnamed__1) at (2.0472, 0.744444, -1.32722);
  \coordinate (v5_unnamed__1) at (1.93568, -0.868631, 1.41611);
  \coordinate (v6_unnamed__1) at (1.30633, -0.0407313, 0.0291552);

  \definecolor{vertexcolor_unnamed__1}{rgb}{ 0 0 0 }

  \tikzstyle{vertexstyle_unnamed__1} = [circle, scale=0.25pt, fill=vertexcolor_unnamed__1,]

  \definecolor{edgecolor_unnamed__1}{rgb}{ 0 0 0 }
  \tikzstyle{edgestyle_unnamed__1} = [thick,color=blue]


  \foreach \i/\k in {1/0,2/0,2/1,3/1,3/2,4/2,5/2,6/1,6/2,6/4,6/5} {
   \draw[edgestyle_unnamed__1] (v\i_unnamed__1) -- (v\k_unnamed__1);
  }

  \foreach \i in {5,0,6,2,1,3,4} {
    \node at (v\i_unnamed__1) [vertexstyle_unnamed__1] {};
  }

\end{tikzpicture}
		\caption{}
	\end{subfigure}
	\begin{subfigure}[b]{0.32\textwidth}
		\centering

\begin{tikzpicture}[scale = 0.25,
	color = {lightgray}]

	\coordinate (v0_unnamed__1) at (-0.5,5);
	\coordinate (v1_unnamed__1) at (1.5,1);
	\coordinate (v2_unnamed__1) at (-1,-5.5);
	\coordinate (v3_unnamed__1) at (-1.5,-1);
	\coordinate (v4_unnamed__1) at (5,0.5);
	\coordinate (v5_unnamed__1) at (0,0);
	\coordinate (v6_unnamed__1) at (-4,0.5);

	\definecolor{vertexcolor_unnamed__1}{rgb}{ 0 0 0 }
	
	\tikzstyle{vertexstyle_unnamed__1_0} = [circle, scale=0.25, fill=vertexcolor_unnamed__1]
	\tikzstyle{vertexstyle_unnamed__1_1} = [circle, scale=0.25, fill=vertexcolor_unnamed__1]
	\tikzstyle{vertexstyle_unnamed__1_2} = [circle, scale=0.25, fill=vertexcolor_unnamed__1]
	\tikzstyle{vertexstyle_unnamed__1_3} = [circle, scale=0.25, fill=vertexcolor_unnamed__1]
	\tikzstyle{vertexstyle_unnamed__1_4} = [circle, scale=0.25, fill=vertexcolor_unnamed__1]
	\tikzstyle{vertexstyle_unnamed__1_5} = [circle, scale=0.25, fill=vertexcolor_unnamed__1]
	\tikzstyle{vertexstyle_unnamed__1_6} = [circle, scale=0.25, fill=vertexcolor_unnamed__1]

	\definecolor{edgecolor_unnamed__1}{rgb}{ 0 0 1 }
	\tikzstyle{edgestyle_unnamed__1} = [thick,color=edgecolor_unnamed__1]
	
	
	\foreach \i/\k in {0/3,0/4,0/5,0/6,1/4,1/5,1/6,2/3,2/4,2/5,2/6,3/5,4/5,5/6} {
		\draw[edgestyle_unnamed__1] (v\i_unnamed__1) -- (v\k_unnamed__1);
	}

	\foreach \i in {0,1,2,3,4,5,6} {
		\node at (v\i_unnamed__1) [vertexstyle_unnamed__1_\i] {};
	}

\end{tikzpicture}
		\caption{}
	\end{subfigure}
	\hfill
	\begin{subfigure}[b]{0.32\textwidth}
		\centering

\begin{tikzpicture}[x  = {(0cm,-1cm)},
                    y  = {(-1cm,0cm)},
                    z  = {(-0.75cm,0.5cm)},
                    xscale = -0.5,
                    yscale = 0.5,
                    rotate = 180,
                    color = {lightgray}]

  \coordinate (v0_unnamed__1) at (1.28308, -1.36686, -0.0554553);
  \coordinate (v1_unnamed__1) at (1.71349, -0.622164, 2.26696);
  \coordinate (v2_unnamed__1) at (-0.584482, -2.64421, -1.06245);
  \coordinate (v3_unnamed__1) at (1.19116, 0.875333, -1.35376);
  \coordinate (v4_unnamed__1) at (0.223707, 0.164358, -0.254224);
  \coordinate (v5_unnamed__1) at (1.08953, 1.09809, 4.06277);
  \coordinate (v6_unnamed__1) at (-3.01128, -2.51031, -1.87869);
  \coordinate (v7_unnamed__1) at (-0.347922, 1.95903, 2.12213);
  \coordinate (v8_unnamed__1) at (-2.6459, -0.0630142, -1.20728);
  \coordinate (v9_unnamed__1) at (2.13039, 1.56554, -2.42122);
  \coordinate (v10_unnamed__1) at (-1.04177, 1.5442, -0.218794);

  \definecolor{vertexcolor_unnamed__1}{rgb}{ 0 0 0 }

  \tikzstyle{vertexstyle_unnamed__1} = [circle, scale=0.25pt, fill=vertexcolor_unnamed__1,]

  \definecolor{edgecolor_unnamed__1}{rgb}{ 0 0 0 }
  \tikzstyle{edgestyle_unnamed__1} = [thick,color=blue]


  \foreach \i/\k in {1/0,2/0,3/0,4/0,4/1,4/2,4/3,5/1,6/2,7/4,7/5,8/4,8/6,9/3,9/4,10/3,10/4,10/7,10/8} {
   \draw[edgestyle_unnamed__1] (v\i_unnamed__1) -- (v\k_unnamed__1);
  }

  \foreach \i in {5,1,7,0,10,4,2,8,3,6,9} {
    \node at (v\i_unnamed__1) [vertexstyle_unnamed__1] {};
  }

\end{tikzpicture}
		\caption{}
	\end{subfigure}
	\hfill
	\begin{subfigure}[b]{0.32\textwidth}
		\centering

\begin{tikzpicture}[scale = 0.5,
	color = {lightgray}]

	\coordinate (v0_unnamed__1) at (-1,1);
	\coordinate (v1_unnamed__1) at (-1,-1);
	\coordinate (v2_unnamed__1) at (0,0);
	\coordinate (v3_unnamed__1) at (-2.5,1);
	\coordinate (v4_unnamed__1) at (-3.5,0);
	\coordinate (v5_unnamed__1) at (-2.5,-1);
	\coordinate (v6_unnamed__1) at (1,-2.5);
	\coordinate (v7_unnamed__1) at (3.7,-2.5);
	\coordinate (v8_unnamed__1) at (3.7,2.5);
	\coordinate (v9_unnamed__1) at (1,2.5);
	\coordinate (v10_unnamed__1) at (4.7,0);

	\definecolor{vertexcolor_unnamed__1}{rgb}{ 0 0 0 }
	
	\tikzstyle{vertexstyle_unnamed__1_0} = [circle, scale=0.25, fill=vertexcolor_unnamed__1]
	\tikzstyle{vertexstyle_unnamed__1_1} = [circle, scale=0.25, fill=vertexcolor_unnamed__1]
	\tikzstyle{vertexstyle_unnamed__1_2} = [circle, scale=0.25, fill=vertexcolor_unnamed__1]
	\tikzstyle{vertexstyle_unnamed__1_3} = [circle, scale=0.25, fill=vertexcolor_unnamed__1]
	\tikzstyle{vertexstyle_unnamed__1_4} = [circle, scale=0.25, fill=vertexcolor_unnamed__1]
	\tikzstyle{vertexstyle_unnamed__1_5} = [circle, scale=0.25, fill=vertexcolor_unnamed__1]
	\tikzstyle{vertexstyle_unnamed__1_6} = [circle, scale=0.25, fill=vertexcolor_unnamed__1]
	\tikzstyle{vertexstyle_unnamed__1_7} = [circle, scale=0.25, fill=vertexcolor_unnamed__1]
	\tikzstyle{vertexstyle_unnamed__1_8} = [circle, scale=0.25, fill=vertexcolor_unnamed__1]
	\tikzstyle{vertexstyle_unnamed__1_9} = [circle, scale=0.25, fill=vertexcolor_unnamed__1]
	\tikzstyle{vertexstyle_unnamed__1_10} = [circle, scale=0.25, fill=vertexcolor_unnamed__1]
	
	\definecolor{edgecolor_unnamed__1}{rgb}{ 0 0 1 }
	\tikzstyle{edgestyle_unnamed__1} = [thick,color=edgecolor_unnamed__1]
	
	
	\foreach \i/\k in {2/0,2/1,3/0,3/2,4/3,5/1,5/2,5/4,6/0,6/1,6/2,7/2,7/6,8/2,9/0,9/1,9/2,9/8,10/7,10/8} {
		\draw[edgestyle_unnamed__1] (v\i_unnamed__1) -- (v\k_unnamed__1);
	}

	\foreach \i in {3,4,0,6,2,5,7,1,9,8,10} {
		\node at (v\i_unnamed__1) [vertexstyle_unnamed__1_\i] {};
	}

\end{tikzpicture}
		\caption{}
	\end{subfigure}
	\begin{subfigure}[b]{0.45\textwidth}
		\centering

\begin{tikzpicture}[x  = {(0cm,0cm)},
                    y  = {(1cm,1cm)},
                    z  = {(1cm,0cm)},
                    scale = 1.1,
                    color = {lightgray}]

  \coordinate (v0_unnamed__1) at (0.469324, -0.937013, 2.57094);
  \coordinate (v1_unnamed__1) at (0.244159, -0.0150771, 0.914403);
  \coordinate (v2_unnamed__1) at (0.0577955, -0.58778, 0.739694);
  \coordinate (v3_unnamed__1) at (-0.444033, 0.377716, -0.745748);
  \coordinate (v4_unnamed__1) at (0.142078, 0.225142, -0.90835);
  \coordinate (v5_unnamed__1) at (-0.469324, 0.937013, -2.57094);

  \definecolor{vertexcolor_unnamed__1}{rgb}{ 0 0 0 }

  \tikzstyle{vertexstyle_unnamed__1} = [circle, scale=0.25pt, fill=vertexcolor_unnamed__1,]

  \definecolor{edgecolor_unnamed__1}{rgb}{ 0 0 0 }
  \tikzstyle{edgestyle_unnamed__1} = [thick,color=blue]
   \tikzstyle{edgestyle_unnamed__2} = [thick,color=red]


  \foreach \i/\k in {1/0,2/0,5/3,5/4} {
   \draw[edgestyle_unnamed__1] (v\i_unnamed__1) -- (v\k_unnamed__1);
  }
  
    \foreach \i/\k in {2/1,3/1,3/2,4/1,4/2,4/3} {
  	\draw[edgestyle_unnamed__2] (v\i_unnamed__1) -- (v\k_unnamed__1);
  }

  \foreach \i in {0,1,2,3,4,5} {
    \node at (v\i_unnamed__1) [vertexstyle_unnamed__1] {};
  }

\end{tikzpicture}
		\caption{}
	\end{subfigure}
	\begin{subfigure}[b]{0.45\textwidth}
		\centering

\begin{tikzpicture}[x  = {(1cm,0cm)},
                    y  = {(0cm,1cm)},
                    z  = {(0cm,0cm)},
                    scale = 0.8, 
                    rotate = -36,
                    color = {lightgray}]

  \coordinate (v0_unnamed__1) at (-1.98293, -1.86514, -0.779692);
  \coordinate (v1_unnamed__1) at (-1.26215, -0.803069, 0.274948);
  \coordinate (v2_unnamed__1) at (0.030463, -0.0464459, -0.138812);
  \coordinate (v3_unnamed__1) at (-0.91512, 1.2551, 0.145107);
  \coordinate (v4_unnamed__1) at (0.62536, -0.81331, 1.17525);
  \coordinate (v5_unnamed__1) at (2.34938, 1.30642, -0.890134);
  \coordinate (v6_unnamed__1) at (1.15499, 0.966445, 0.213329);

  \definecolor{vertexcolor_unnamed__1}{rgb}{ 0 0 0 }

  \tikzstyle{vertexstyle_unnamed__1} = [circle, scale=0.25pt, fill=vertexcolor_unnamed__1,]

  \definecolor{edgecolor_unnamed__1}{rgb}{ 0 0 0 }
  \tikzstyle{edgestyle_unnamed__1} = [thick,color=blue]
    \tikzstyle{edgestyle_unnamed__2} = [thick,color=red]


  \foreach \i/\k in {1/0,2/0,5/2,6/5} {
   \draw[edgestyle_unnamed__1] (v\i_unnamed__1) -- (v\k_unnamed__1);
  }
  
    \foreach \i/\k in {2/1,3/1,3/2,4/1,4/2,6/2,6/3,6/4} {
  	\draw[edgestyle_unnamed__2] (v\i_unnamed__1) -- (v\k_unnamed__1);
  }

  \foreach \i in {4,1,6,3,2,0,5} {
    \node at (v\i_unnamed__1) [vertexstyle_unnamed__1] {};
  }

\end{tikzpicture}
		\caption{}
	\end{subfigure}
	\begin{subfigure}[b]{0.45\textwidth}
		\centering

\begin{tikzpicture}[x  = {(1cm,0cm)},
                    y  = {(0cm,1cm)},
                    z  = {(0.5cm,-0.5cm)},
                    scale = 0.9,
                    color = {lightgray}]

  \coordinate (v0_unnamed__1) at (-3.33728, 1.03114, -0.112738);
  \coordinate (v1_unnamed__1) at (-1.14109, -0.191281, -0.962807);
  \coordinate (v2_unnamed__1) at (-0.89173, 0.819343, 0.894069);
  \coordinate (v3_unnamed__1) at (-0.348499, -1.05966, 0.623526);
  \coordinate (v4_unnamed__1) at (0.348498, 1.05966, -0.623526);
  \coordinate (v5_unnamed__1) at (-0.409256, -1.65919, -3.04844);
  \coordinate (v6_unnamed__1) at (0.409258, 1.65919, 3.04844);
  \coordinate (v7_unnamed__1) at (0.89173, -0.819343, -0.894069);
  \coordinate (v8_unnamed__1) at (1.14109, 0.191281, 0.962807);
  \coordinate (v9_unnamed__1) at (3.33728, -1.03114, 0.112737);

  \definecolor{vertexcolor_unnamed__1}{rgb}{ 0 0 0 }

  \tikzstyle{vertexstyle_unnamed__1} = [circle, scale=0.25pt, fill=vertexcolor_unnamed__1,]

  \definecolor{edgecolor_unnamed__1}{rgb}{ 0 0 0 }
  \tikzstyle{edgestyle_unnamed__1} = [thick,color=blue]
  \tikzstyle{edgestyle_unnamed__2} = [thick,color=red]


  \foreach \i/\k in {1/0,2/0,5/1,6/2,7/5,8/6,9/7,9/8} {
   \draw[edgestyle_unnamed__1] (v\i_unnamed__1) -- (v\k_unnamed__1);
  }
  
  \foreach \i/\k in {2/1,3/1,3/2,4/1,4/2,7/1,7/3,7/4,8/2,8/3,8/4,8/7} {
  	\draw[edgestyle_unnamed__2] (v\i_unnamed__1) -- (v\k_unnamed__1);
  }

  \foreach \i in {6,8,2,3,9,0,4,7,1,5} {
    \node at (v\i_unnamed__1) [vertexstyle_unnamed__1] {};
  }

\end{tikzpicture}
		\caption{}
	\end{subfigure}
	\begin{subfigure}[b]{0.45\textwidth}
		\centering

\begin{tikzpicture}[x  = {(1cm,0cm)},
                    y  = {(0cm,1cm)},
                    z  = {(0cm,0cm)},
                    scale = 0.8,
                    rotate = 180,
                    color = {lightgray}]

  \coordinate (v0_unnamed__1) at (-0.0656107, 0.314487, 2.42653);
  \coordinate (v1_unnamed__1) at (2.34921, -0.18671, 0.129519);
  \coordinate (v2_unnamed__1) at (-1.66312, -1.66117, 0.212124);
  \coordinate (v3_unnamed__1) at (-0.603481, 1.85135, 3.73289);
  \coordinate (v4_unnamed__1) at (0.477091, -1.34081, -0.759276);
  \coordinate (v5_unnamed__1) at (3.84053, 0.531654, -1.0017);
  \coordinate (v6_unnamed__1) at (-0.49001, 1.41562, 1.46705);
  \coordinate (v7_unnamed__1) at (-3.3046, -2.09404, -0.8546);
  \coordinate (v8_unnamed__1) at (1.06946, 0.456336, -1.10416);
  \coordinate (v9_unnamed__1) at (-1.1497, -0.359162, -1.05847);
  \coordinate (v10_unnamed__1) at (-0.459779, 1.07245, -3.18991);

  \definecolor{vertexcolor_unnamed__1}{rgb}{ 0 0 0 }

  \tikzstyle{vertexstyle_unnamed__1} = [circle, scale=0.25pt, fill=vertexcolor_unnamed__1,]

  \definecolor{edgecolor_unnamed__1}{rgb}{ 0 0 0 }
  \tikzstyle{edgestyle_unnamed__1} = [thick,color=blue]
  \tikzstyle{edgestyle_unnamed__2} = [thick,color=red]


  \foreach \i/\k in {5/1,0/3,6/3,7/2,8/5,9/6,9/7,10/8,10/9} {
   \draw[edgestyle_unnamed__1] (v\i_unnamed__1) -- (v\k_unnamed__1);
  }
  
  \foreach \i/\k in {1/0,2/0,4/1,4/2,6/0,8/1,8/4,8/6,9/2,9/4,9/6,9/8} {
  	\draw[edgestyle_unnamed__2] (v\i_unnamed__1) -- (v\k_unnamed__1);
  }

  \foreach \i in {3,0,6,2,1,4,7,5,9,8,10} {
    \node at (v\i_unnamed__1) [vertexstyle_unnamed__1] {};
  }

\end{tikzpicture}
		\caption{}
	\end{subfigure}
	\caption{Tight spans of the non-split coarsest subdivisions of the hypersimplex $\Delta(2,6)$.
          The top six are $2$-dimensional, while the bottom four are $3$-dimensional.
          The $3$-dimensional cases feature a unique $3$-cell, whose edges are marked red
        }
	\label{fig:Delta26:tightspans}
\end{figure}

Let us illustrate the use of Proposition~\ref{prop:dg:contraction} beyond complete graphs by example.
Figure~\ref{fig:Delta26:tightspans} displays (the graphs of) the tight spans of regular subdivisions of the hypersimplex $\Delta(2,6)$, as defined in Section \ref{sec:hypersimplices} below.
Take the graph (A) in the first row.
No edge can be contracted without producing multiedges and the only way to arrive at a simple graph by contraction is to contract everything.
Hence this is the tight span of a coarsest subdivision, even though the graph is not a complete graph.
The graph (C) in the first row is different.
It has two distinguished halves and it is possible to arrive at a simple graph by contracting just one of these halves, leaving the other half untouched.
From Proposition~\ref{prop:dg:contraction} alone we cannot detect that the subdivision is coarsest.

\section{Subdivisions of Hypersimplices}
\label{sec:hypersimplices}
Our main objects of study are regular subdivisions of the hypersimplex $\Delta(k,n)$, where $1\leq k \leq n$.
The latter is the polytope
\[
  \conv\bigl\{ e_X \bigmid X\in \tbinom{[n]}{k} \bigr\} \ = \ \bigl[0,1\bigr]^n \cap \bigl\{ x\in\RR^n \bigmid \sum x_i = k \bigr\} \enspace ,
\]
where $e_X = \sum_{i\in X} e_i$, and $e_1,e_2,\dots,e_n$ are the standard basis vectors in $\RR^n$.
Further, $\tbinom{[n]}{k}$ is the set of $k$-element subsets of $[n]=\{1,2,\dots,n\}$.
The hypersimplices live in codimension one; to make them full-dimensional, as we had assumed in the previous section, we can project out, e.g., the last coordinate.

The hypersimplices are \emph{lattice polytopes}, i.e., their vertices have integral coordinates.
Exchanging the zero and the one coordinates induces a \emph{lattice equivalence} between $\Delta(k,n)$ and $\Delta(n-k,n)$.
That is to say, there is an affine isomorphism from one polytope to the other which preserves the integer lattice.
The hypersimplex $\Delta(1,n)\cong\Delta(n-1,n)$ is an $(n{-}1)$-dimensional simplex.

\begin{example}
  \label{exmp:Delta24}
  The hypersimplex $\Delta(2,4)$ is lattice equivalent to the regular octahedron, embedded in $\RR^4$.
  It has three (regular) triangulations, each of which comprises four tetrahedra around an edge spanning two antipodal vertices.
  Modulo linealities, the secondary fan is the normal fan of a triangle, and the three rays correspond to the three splits of $\Delta(2,4)$ into pairs of Egyptian pyramids.
\end{example}

Hypersimplices and their subdivisions occur naturally in tropical geometry.
Here we sketch how.
Let $\omega$ be a lifting function on $\Delta(k,n)$.
Then the vertex-edge graph of the regular subdivision $\Delta(k,n)^\omega$ contains the vertex-edge graph of $\Delta(k,n)$ as a subgraph.
Now the function $\omega$ is a \emph{$(k,n)$-uniform tropical Plücker vector} if and only if these two graphs coincide.
For experts in tropical geometry this is a known result; see, e.g., \cite[Theorem 10.36]{ETC}.
Readers who are new to this topic may read the above as a definition.
The tight span of the regular subdivision induced by a tropical Plücker vector is a \emph{tropical linear space}, up to minor boundary effects; see \cite[Example 10.56]{ETC}.
The \emph{Dressian} $\Dr(k,n)$ is the subfan of the secondary fan $\Sigma(k,n)$ comprised of the $(k,n)$-uniform tropical Plücker vectors.
For instance, the Dressian $\Dr(2,4)$ is comprised of the three splits of $\Delta(2,4)$; see Example~\ref{exmp:Delta24}.
The Dressian contains the \emph{tropical Grassmannian} $\TGr(k,n)$ of tropicalized linear spaces, as a subset but not necessarily as a subfan.

The second result which is crucial for our present work is a bound on the combinatorial complexity of tropical linear spaces.
\begin{theorem}[{Speyer \cite[Theorem~6.1]{Speyer:2008}}]
  \label{thm:spread}
  Let $\omega$ be a lifting function of $\Delta(k,n)$.
  If $\omega$ is contained in the Dressian $\Dr(k,n)$, then the spread of $\Delta(k,n)^\omega$ is at most $\tbinom{n-2}{k-1}$.
\end{theorem}
Speyer's result above has an interesting direct consequence for the rays of the tropical Grassmannian.
\begin{corollary}\label{cor:spread}
  Let $\omega$ be a ray of the tropical Grassmannian $\TGr(k,n)$.
  Then the spread of $\Delta(k,n)^\omega$ is at most $\tbinom{n-2}{k-1}-(k-1)\cdot(n-k-1)+1$.
\end{corollary}
\begin{proof}
  The tropical Grassmannian $\TGr(k,n)$, modulo its lineality space, is a polyhedral fan of pure dimension $(k-1)\cdot(n-k-1)$ (see \cite[Corollary 3.1]{SpeyerSturmfels04} and \cite[\S4.3]{Tropical+Book}).
  Any maximal flag starting with the ray $\omega$ gives rise to a sequence of regular subdivisions.
  The claim now follows from Theorem~\ref{thm:spread} because all maximal cones of $\TGr(k,n)$ share the same dimension, and since the spread of the subdivisions of the flag increases by at least one in each step.
\end{proof}
Plugging in $k=2$ into the bound above yields $2$, for any $n\geq 4$.
So the corollary generalizes the known fact that each ray of the tropical Grassmannian $\TGr(2,n)$ is a split.
In particular, the given bound is tight.

The symmetric group $\Sym(n)$ acts transitively on the vertices of the hypersimplex $\Delta(k,n)$ for arbitrary $1\leq k \leq n$.
For $n=2k$ there is an additional involution which exchanges zeros with ones coordinatewise.
In that case the full automorphism group is isomorphic to the semidirect product $\Sym(n) \rtimes \ZZ/2$.
The action is (affinely) linear, whence it induces an action on the set of all (regular) subdivisions.

\begin{example}
  \label{exmp:Delta25}
  Modulo linealities, the Dressian $\Dr(2,5)$ is a two-dimensional fan with ten rays and 15 maximal cones.
  Its intersection with the unit sphere is isomorphic to the Petersen graph, seen as a one-dimensional polytopal complex \cite[Figure 10.7]{ETC}.
  The secondary fan $\Sigma(2,5)$ has another ten rays; see Figure~\ref{fig:Dr25}.
  Up to lineality, the dimension of the secondary fan equals five.
  This partition of the rays of $\Sigma(2,5)$ into splits and non-splits agrees with the partition into $\Sym(5)$-orbits.
  Each non-split ray has spread five.
  Since that value exceeds Speyer's bound $\tbinom{5-2}{2-1}=3$ from Theorem~\ref{thm:spread} it is clear that the non-splits cannot lie in $\Dr(2,5)$.
  Of course, Corollary~\ref{cor:spread} applies, too, as $\TGr(2,5)=\Dr(2,5)$.
  See also Example~\ref{exmp:Delta25-rays} below.
\end{example}

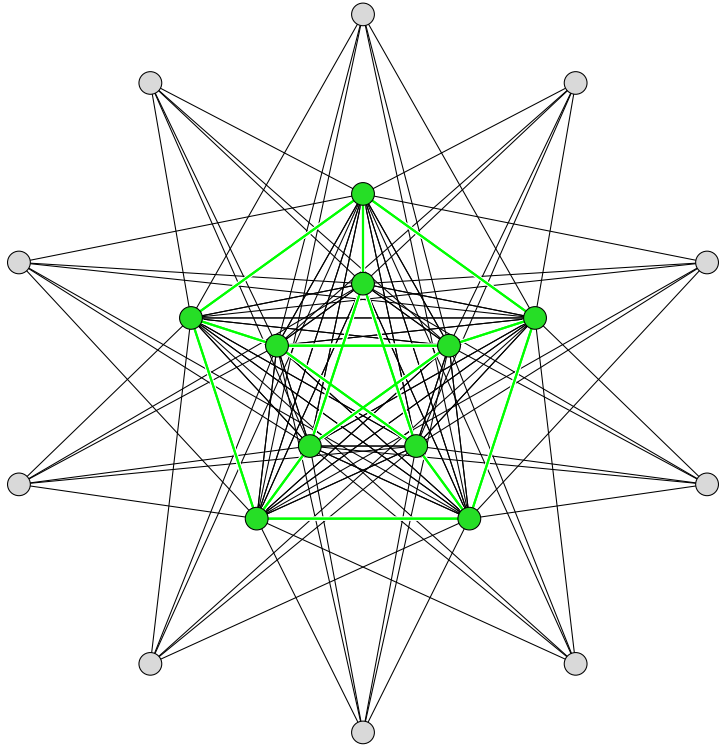
\begin{figure}
  \centering

  \begin{tikzpicture}[x={(1.7cm,0cm)},y={(0cm,1.7cm)},
	secondary/.style={circle,minimum size=3mm,fill=black!15,draw},
        dressian/.style={circle,minimum size=3mm,fill=black!15!green!85,draw}, scale=0.7] 
	
	\coordinate (p0) at (0*72+18:2);
	\coordinate (p4) at (1*72+18:2);
	\coordinate (p1) at (2*72+18:2);
	\coordinate (p5) at (3*72+18:2);
	\coordinate (p3) at (4*72+18:2);
	\coordinate (p13) at (0*72+18:1);
	\coordinate (p12) at (1*72+18:1);
	\coordinate (p7) at (2*72+18:1);
	\coordinate (p14) at (3*72+18:1);
	\coordinate (p6) at (4*72+18:1);

	\coordinate (p2) at (6*36+18:4);
	\coordinate (p8) at (2*36+18:4);
	\coordinate (p9) at (0*36+18:4);
	\coordinate (p10) at (8*36+18:4);
	\coordinate (p11) at (4*36+18:4);
	\coordinate (p15) at (9*36+18:4);
	\coordinate (p16) at (1*36+18:4);
	\coordinate (p17) at (3*36+18:4);
	\coordinate (p18) at (5*36+18:4);
	\coordinate (p19) at (7*36+18:4); 
	
	
	\draw \foreach \i in {1, 2, 5, 6, 7, 8, 10, 11, 12, 14, 15, 16} {(p0) -- (p\i) }; 
	\draw \foreach \i in {0, 2, 3, 6, 8, 9, 10, 12, 13, 14, 17, 18}{(p1) -- (p\i) };
	\draw \foreach \i in {1, 2, 4, 7, 8, 9, 11, 12, 13, 14, 15, 19}{(p3) -- (p\i) };
	\draw \foreach \i in {2, 3, 5, 6, 7, 9, 10, 11, 13, 14, 16, 17}{(p4) -- (p\i) };
	\draw \foreach \i in {0, 4, 6, 7, 8, 9, 10, 11, 12, 13, 18, 19}{(p5) -- (p\i) };
	\draw \foreach \i in {0, 1, 2, 4, 5, 9, 13, 14, 15, 16, 18, 19}{(p6) -- (p\i) };
	\draw \foreach \i in {0, 2, 3, 4, 5, 8, 12, 14, 16, 17, 18, 19}{(p7) -- (p\i) };
	\draw \foreach \i in {0, 1, 3, 5, 7, 9, 11, 13, 15, 16, 17, 18}{(p12) -- (p\i) };
	\draw \foreach \i in {1, 3, 4, 5, 6, 8, 10, 12, 15, 16, 17, 19}{(p13) -- (p\i) };
	\draw \foreach \i in {0, 1, 3, 4, 6, 7, 10, 11, 15, 17, 18, 19}{(p14) -- (p\i) };

	\draw[white, ultra thick] \foreach \i in {3, 4, 13} {(p0) -- (p\i) };
	\draw[white, ultra thick]  \foreach \i in {4, 5, 7} {(p1) -- (p\i) };
	\draw[white, ultra thick]  \foreach \i in {0, 5, 6} {(p3) -- (p\i) };
	\draw[white, ultra thick]  \foreach \i in {0, 1, 12} {(p4) -- (p\i) };
	\draw[white, ultra thick]  \foreach \i in {1, 3, 14} {(p5) -- (p\i) };
	\draw[white, ultra thick] \foreach \i in {3, 7, 12} {(p6) -- (p\i) };
	\draw[white, ultra thick] \foreach \i in {1, 6, 13} {(p7) -- (p\i) };
	\draw[white, ultra thick] \foreach \i in {4, 6, 14} {(p12) -- (p\i) };
	\draw[white, ultra thick] \foreach \i in {0, 7, 14} {(p13) -- (p\i) };
	\draw[white, ultra thick] \foreach \i in {5, 12, 13} {(p14) -- (p\i) };

        \draw[green, thick] \foreach \i in {3, 4, 13} {(p0) -- (p\i) };
	\draw[green, thick]  \foreach \i in {4, 5, 7} {(p1) -- (p\i) };
	\draw[green, thick]  \foreach \i in {0, 5, 6} {(p3) -- (p\i) };
	\draw[green, thick]  \foreach \i in {0, 1, 12} {(p4) -- (p\i) };
	\draw[green, thick]  \foreach \i in {1, 3, 14} {(p5) -- (p\i) };
	\draw[green, thick] \foreach \i in {3, 7, 12} {(p6) -- (p\i) };
	\draw[green, thick] \foreach \i in {1, 6, 13} {(p7) -- (p\i) };
	\draw[green, thick] \foreach \i in {4, 6, 14} {(p12) -- (p\i) };
	\draw[green, thick] \foreach \i in {0, 7, 14} {(p13) -- (p\i) };
	\draw[green, thick] \foreach \i in {5, 12, 13} {(p14) -- (p\i) };

	\tiny
	
	\foreach \i in {2,8,9,10,11,15,16,17,18,19} \node[secondary] at (p\i) {}; 
        \foreach \i in {0,1,3,4,5,6,7,12,13,14} \node[dressian] at (p\i) {}; 
	
\end{tikzpicture}


  \caption{Dressian $\Dr(2,5)$ embedded into (the graph of) $\Sigma(2,5)$}
  \label{fig:Dr25}
\end{figure}

\section{Volumes and Eulerian Numbers}
\label{sec:eulerian}

Our goal is to identify and study coarsest subdivisions of the hypersimplices $\Delta(k,n)$, for $2\leq k \leq n-2$.
The splits are known; there are precisely $\lfloor n/2 \rfloor-1$ orbits with respect to the $\Sym(n)$ action; see also \cite[Theorem 5.3]{HerrmannJoswig:2008}.
Furthermore, Herrmann constructed \enquote{3-splits} of the hypersimplices, which also form rays of the Dressians \cite[Theorem 6.5]{MR2739494}.
In our analysis we will encounter the Eulerian numbers, which are known to express the volumes of the hypersimplices \cite{Foata:1977,Stanley:1977}.

\begin{remark}\label{rem:ordering}
  In the sequel we will assume that the $\tbinom{n}{k}$ vertices of $\Delta(k,n)$ are in descending lexicographic order.
  For instance, the vertices of $\Delta(2,4)$ are ordered $1100$, $1010$, $1001$, $0110$, $0101$, $0011$. Denote by $V_{k,n}$ the vertices of $\Delta(k,n)$ in our ordering. 
\end{remark}

For $1\leq i \leq \tbinom{n}{k}$ we define
\begin{equation}\label{lambda}
  \lambda_{k,n}(i) \ = \
  \begin{cases}
    1 & 1 \leq i\leq n-k;\\
    0 & \text{otherwise,}
  \end{cases}
\end{equation}
which yields a lifting function on $\Delta(k,n)$.
For instance, $\lambda_{2,4}(1100)=\lambda_{2,4}(1010)=1$, and $\lambda_{2,4}$ is zero on the remaining four vertices of $\Delta(2,4)$.
We abbreviate $\lambda=\lambda_{k,n}$ if the parameters are obvious from the context.
Throughout we assume $2\leq k \leq n-2$.
The tight span of the subdivision induced by $\lambda_{2,6}$ is shown in Figure~\ref{fig:Delta26:tightspans}(B).


\begin{proposition}
  \label{prop:lambda}
  The regular subdivision $\Delta(k,n)^\lambda$ is coarsest.
\end{proposition}
\begin{proof}
  We show that the secondary cone of $\Delta(k,n)^\lambda$ is a ray of the secondary fan.
  To this end, we will prove that the secondary cone of $\Delta(k,n)^\lambda$, which lies in $\RR^{\binom{n}{k}}$, is defined by $\binom{n}{k}-1$ independent linear equations, modulo linealities.
  
  First, observe that $\lambda$ is nonnegative.
  Consequently, the set 
  \begin{equation}\label{eq:lambda:cell-C}
    C \ = \ \conv \bigl\{ x \in V_{k,n} \bigmid \lambda(x)=0 \bigr\}
  \end{equation}
  is a projection of a face of the lifted polytope.
  That is, $C$ is a cell of the subdivision $\Delta(k,n)^\lambda$, and it is full-dimensional.
  Its vertices are labeled ${n-k+1}, \dots, {\binom{n}{k}}$.

  We want to describe the secondary cone.
  To factor out the lineality space, we consider only those height vectors $z$ for which all vertices of $C$ are lifted to 0.
  The pointed secondary cone thus satisfies $z_i = 0$ for every $i \in \{n-k+1, \dots, \binom{n}{k}\}$.
  These form $\binom{n}{k} - n+k$ linear equalities, and we are left to find $n-k-1$ more linear equalities. 
  
  To this end we want to identify a second maximal cell, which contains the vertices labeled $1, \dots, n-k$ that do not belong to $C$.
  Consider the linear function
  \begin{equation} \label{inequality}
    f(x) = \lambda(x) + 2k-2 - 2(x_1+\ldots+x_{k-1})-x_k-\ldots-x_{n-1} \enspace.
  \end{equation}
  We will now show that $f(x) \geq 0$ holds for every vertex $x$ of $\Delta(k,n)$, and $f(x)=0$ holds for the first $n-k$ vertices.
  If $x$ is one of these, its first $k-1$ coordinates read $x_1 = \dots = x_{k-1} = 1$.
  Since $\sum x_i=k$ we infer that there is exactly one index $j$ with $k \leq j \leq n-1$ such that $x_j=1$.
  Furthermore we have $\lambda(x)=1$ and thus $f(x) = 1 + 2k-2 - 2(k-1) - 1 = 0$, as claimed.
  Next we will show that, if $f$ is nonnegative for each vertex of $\Delta(k,n)$.
  It remains to verify this property for the vertices labeled $n-k+1,\dots,\tbinom{n}{k}$.

  If $x$ is the $(n-k{+}1)$st vertex, we have $\lambda(x)=0$, $x_1 = \dots = x_{k-1} = x_{n} = 1$ and all other coordinates vanish. Hence $f(x) = 0$.

  Finally, we consider the case where $x$ is one of the vertices labeled $n-k+2, \dots, \tbinom{n}{k}$.
  Then $\lambda(x)=0$ and $2(x_1+\dots+x_{k-1})+x_{k+1}+\dots+x_{n-1} < 2k-2$, whence $f$ is strictly positive.
  As for the maximal cell $C$ in \eqref{eq:lambda:cell-C}, we conclude that the set
  \begin{equation}\label{eq:lambda:cell-C-prime}
    C' \ = \ \conv \bigl\{ x \in V_{k,n} \bigmid f(x)=0 \bigr\}
  \end{equation}
  is a cell of the subdivision $\Delta(k,n)^\lambda$, which does not need to be maximal however.
  
  Let $C''$ be some maximal cell containing $C'$.
  By construction, $C''$ contains the first $n-k$ vertices of $\Delta(k,n)$.
  Since these vertices are not contained in $C$ and are affinely independent, the cell $C''$ contributes at least $n-k-1$ linear equations to the secondary cone of $S$, which are independent of the $\binom{n}{k}-n+k$ equations found before.
  Hence, the dimension of the secondary cone of $\Delta(k,n)^\lambda$ is at most one, and thus exactly one. 
\end{proof}

In the sequel we will keep the three cells $C$, $C'$ and $C''$ gotten from \eqref{eq:lambda:cell-C} and \eqref{eq:lambda:cell-C-prime}.
The cells $C$ and $C''$ are maximal cells, $C'$ may not be maximal, and $C''$ may not uniquely determined.
We call $C$ the \emph{large cell}.

\begin{proposition}\label{prop:non-matroidal}
  The large maximal cell $C$, with $\tbinom{n}{k}-n+k$ vertices, is not a matroid polytope.
  Consequently, $\lambda$ does not lie in the Dressian $\Dr(k,n)$.
\end{proposition}
\begin{proof}
  It suffices to show that the polytope $C$ has an edge whose direction is not parallel to $e_i - e_j$ for some $i,j$; e.g., see \cite[Theorem 10.4]{ETC}.
  For example, the line through the two adjacent vertices labeled $n-k+1$ and $n-k+2$ has direction $e_{k-1}-e_{k}-e_{k+1}+e_{\binom{n}{k}}$.
\end{proof}

Before we continue, we recall a known relationship between the Eulerian numbers and the volumes of hypersimplices.
The \emph{Eulerian numbers} $A(n,k)$ are given by the recursion
\begin{equation}\label{eulerian:recursion}
  A(n,k)\ =\ k\cdot A(n-1,k) + (n-k+1)\cdot A(n-1,k-1) \enspace,
\end{equation}
with $A(1,1)=1$ and $A(1,k)=0$ for $k\not= 1$.
The formula for the lattice volume
\begin{equation}\label{eq:volume}
  \vol(\Delta(k,n)) \ = \ A(n-1,k)
\end{equation}
was observed by Foata \cite{Foata:1977} and Stanley \cite{Stanley:1977}, who mention that this implicitly occurs already in work of Laplace; see also \cite[Proposition 6.3.16]{Triangulations}.
The Eulerian numbers can be found in The On-Line Encyclopedia of Integer Sequences at \url{https://oeis.org/A008292}.
Note that the lattice volume of $\Delta(k,n)$ agrees with the normalized volume of the projection obtained from omitting one coordinate.


In the special case $k=2$ we obtain the following particularly concise result.
This will become instrumental in Section~\ref{sec:metric} on finite metric spaces below.

\begin{theorem}
  \label{thm:lambda}
  The regular subdivision $\Delta(2,n)^\lambda$ is a non-matroidal coarsest subdivision of spread $n$.
\end{theorem}

\begin{proof}
  We continue our analysis above, specializing to the case $k=2$.
  As $n>2$, the height function $\lambda$ does not induce a split, and so it does not lie in the Dressian $\Dr(2,n)$.
  Of course, this we already knew from Proposition~\ref {prop:non-matroidal}.

  \underline{Claim A:} The spread of the regular subdivision $\Delta(2,n)^\lambda$ is at least $n$.
  So we need to exhibit $n-2$ maximal cells other than $C$ and $C''$.
  Consider the $n-2$ linear functions
  \begin{equation}\label{eq:gamma-i}
    g_i(x) = \lambda(x) + 1 - x_1 - x_i  \enspace ,
  \end{equation}
  where $i \in \{2, \ldots, n-1 \}$.
  It is easy to see that $g_i$ is nonnegative on the vertices of $\Delta(2,n)$, and the proof is analogous to the analysis of $f$ in the proof of Proposition~\ref{prop:lambda}. 
  	
  We will now show that, for every $i$, there are exactly $n$ vertices of $\Delta(2,n)$ for which $g_i$ attains zero.
  With $\lambda(x)=1$ the only vertex $x$ with $g_i(x)=0$ is the one defined by $x_1 = x_i = 1$.
  We are left with the case $\lambda(x)=0$.
  Then the vertex defined by $x_1 = x_n = 1$ is one solution for which $g_i$ vanishes.
  All other solutions are those vertices which have zero as the first coordinate, and two ones among the remaining $n-1$ coordinates, one of which must be $i$ (giving $n-2$ choices).
  So the total number of solutions adds up to $n$, as desired.
  Consequently, these $n$ vertices form a cell, which we denote $C_i$.
  By construction $C_i \neq C_j$ for distinct $i$ and $j$.
  The cells $C_i$ are maximal because their vertices form a matrix of full rank $n$. 
  Clearly $C_i\neq C''$ because $g_i$ does not vanish on all of the first $n-2$ vertices.
  Furthermore, $g_i$ neither vanishes on all the vertices of $C$, which guarantees $C_i\neq C$.
  This concludes the proof that the spread is at least $n$.
        
  It remains to show that $\Delta(2,n)^\lambda$ does not have any maximal cells other than $C$, $C''$, and $C_1,\dots,C_{n-2}$.
  Recall that the cell $C_i$ is defined via the function $g_i$ from \eqref{eq:gamma-i}.

  \underline{Claim B:} The maximal cells $C$ and $C''$ are pyramids of height one over the hypersimplex $\Delta_{2,n-1}$.
  First, observe that both cells contain the vertex $e_1+e_n$, which is labeled $n-1$.
  This vertex forms the apex of both pyramids.
  Now $C$ consists of the vertex $e_1+e_n$ and the vertices with labels $n,\ldots,\binom{n}{2}$.
  Since $k=2$, the latter points are exactly those with first coordinate equal to $0$.
  Hence they form the vertices of the first contraction facet, and this is lattice isomorphic to $\Delta(2,n-1)$.

  For $C''$ a symmetric argument works.
  The reason is that $C''$ comprises the apex and all vertices that end with a $0$.
  The lattice height of the apex $e_1+e_n$ over both bases equals one.
  This argument concludes the proof of Claim~B.

  \underline{Claim C:} The cells $C$ and $C''$ both have lattice volume $A(n-2,2)$.
  This follows immediately from Claim B and the known volume formula \eqref{eq:volume}.

  Now we are ready to complete the proof of the theorem.
  Each maximal cell $C_1, \dots, C_{n-1}$ has lattice volume at least one.
  This observation and the recursion formula \eqref{eulerian:recursion} imply that the total volume of the $n$ maximal cells $C,C'',C_1,\dots,C_{n-2}$ is at least
  \[
    2\cdot A(n-2,2) + n-2 \ = \ 2\cdot A(n-2,2) + (n-2)\cdot A(n-2,1) \ = \ A(n-1,2) \enspace .
  \]
  From \eqref{eq:volume} we know that the latter value equals the lattice volume of $\Delta(2,n)$.
  This entails that there is no room for more maximal cells.
  Furthermore, the argument also tells us that the maximal cells $C_1,\dots,C_{n-2}$ have lattice volume one and thus must be simplices.
\end{proof}

\begin{example}\label{exmp:Delta25-rays}
  Theorem~\ref{thm:lambda} suffices to fully explain the secondary fan $\Sigma(2,5)$ discussed in Example~\ref{exmp:Delta25}.
  The non-split rays of $\Sigma(2,5)$ shown in Figure~\ref{fig:Dr25} form one orbit with respect to the natural $\Sym(5)$-action.
  So, each ray of $\Sigma(2,5)$ is either a split or equivalent to $\lambda_{2,5}$.
\end{example}

We define a second lifting function
\begin{equation}\label{eulerian:height}
   \kappa_{k,n}(i)\ =\ \begin{cases}
      1 & 1\leq i\leq \binom{n-1}{k-1}-1\\
      0 & \text{otherwise.}
   \end{cases}
\end{equation}
For our ordering of the vertices of $\Delta(k,n)$ this means that all the vertices with a leading one are lifted to height one, except for the last one of these.
\begin{observation}\label{obs:lambda=kappa}
  As $n-2 = \tbinom{n-1}{2-1}-1$, we have $\lambda_{2,n}=\kappa_{2,n}$ for every $n$.
\end{observation}
Again we abbreviate $\kappa=\kappa_{k,n}$ if the parameters are obvious from the context. We are now going to analyze the properties of the subdivision induced by $\kappa_{k,n}$ for general $k$.  In particular, in the rest of the section we will show that the subdivision is coarsest and that the cell volumes are Eulerian numbers.

Let $\tilde{\Delta}(k,n)$ be the lifted hypersimplex with vertices $V_{k,n}\mid
\kappa_{k,n}$ 
. Let $c$ the vertex labeled $\binom{n-1}{k-1}$ namely 
\[ c = e_1+ \sum_{j = n-k+2}^{n} e_j \enspace, \]
i.e., a leading one, $n-k$ zeros, and $k-1$ ones. We use the letter $c$ for ``center'', since it will become clear that $c$ is contained in every cell of the subdivision $\Delta(k,n)^\kappa$.
The lifted point $\tilde{c}$ is
just $c$ with a trailing zero added.

\begin{lemma}\label{eulerian:lemma:small_facets}
	Consider the affinely linear functions
	\[
	f_i(x)\ =\ 1-x_1-x_i+x_{n+1} \enspace,
	\]
	where $i\in\{2,\ldots,n-k+1\}$,
	on $\RR^{n+1}$, the space where $\tilde{\Delta}(k,n)$ lives. Then
	\begin{enumerate}
		\item the $f_i$ cut out lower facets of $\tilde{\Delta}(k,n)$;
		\item these facets are isomorphic to lattice height one pyramids over 
		$\Delta(k-1,n-1)$.
	\end{enumerate}
	There are exactly $n-k$ such facets.
\end{lemma}
\begin{proof}
   We start by analyzing the values of $f_i$ on the various lifted points.
	\begin{itemize}
      \item For the lifted center $\tilde{c}$ we have $f_i(\tilde{c})=0$, since
         $\tilde{c}_1=1$ and $\tilde{c}$ has zeros at positions
         $i\in\{2,\ldots,n-k+1\}$, and $\tilde{c}_{n+1}=0$, since it is lifted
         to height 0.
      \item We have $f_i\ge 0$ on the points lifted to height $1$, since the height is added in the last coordinate and hence we have $x_{n+1}=1$ in the definition of $f_i$, while the other variables are $0$ or $1$.
		\item Lastly, we have $f_i\ge 0$ on the points lifted to height $0$. The point
		$\tilde{c}$ was discussed above, all other points will not have a leading one, and
		hence $x_1$ is zero.
	\end{itemize}
	In particular, this implies that the $f_i$ cut out lower facets. We have not yet
	deduced the dimension, but this will become clear from the remainder of this
	proof.
	
	Each lower facet defined by $f_i$
	consists of the points where $f_i$ evaluates to $0$. These are: 
	\begin{itemize}
		\item the point $\tilde{c}$ as discussed above;
		\item the points with leading one, i.e. $x_1=1$, and $x_i=1$, since these also have
		height $x_{n+1}=1$;
		\item The points with leading zero and $x_i=1$, since these are lifted to
		height $x_{n+1}=0$.
	\end{itemize}
	In terms of $\Delta(k,n)$ the projected set of points consists of the
	contraction facet $x_i=1$, which is isomorphic to the hypersimplex $\Delta(k-1,n-1)$
	and the point $c$. The height of $c$ above this facet can be determined by
	taking the scalar product with the $i$-th unit vector. Hence these facets
	are pyramids of height one over $\Delta(k-1,n-1)$, so their lattice volume
	must be equal to the lattice volume of $\Delta(k-1,n-1)$.
\end{proof}

\begin{lemma}\label{eulerian:lemma:large_facets}
	Consider the affinely linear functions
	\[
	g_i(x)\ =\ -x_1+x_i+x_{n+1} \enspace,
	\]
	where $i\in\{n-k+2,\ldots,n\}$,
	on $\RR^{n+1}$, the space where $\tilde{\Delta}(k,n)$ lives. Then
	\begin{enumerate}
		\item The $g_i$ cut out lower facets of $\tilde{\Delta}(k,n)$.
		\item These facets are isomorphic to lattice height one pyramids over 
		$\Delta(k,n-1)$.
	\end{enumerate}
	There are exactly $k-1$ such facets,
\end{lemma}
\begin{proof}
	Positivity of $g_i$ on $\tilde{\Delta}(k,n)$ follows along the lines of the
	proof for $f_i$.
	
	The points cut out by $g_i$ are the points of $\Delta(k,n)$ such that
	$x_i=0$ and $c$. The contraction facet with $x_i=0$ is isomorphic to $\Delta(k,n-1)$,
	together with $c$ we have a pyramid over it.
\end{proof}

\begin{proposition}
  \label{prop:kappa}
  The subdivision $\Delta(k,n)^\kappa$ comprises exactly $n$ maximal cells.
  Of these cells, $k$ are pyramids of height one over $\Delta(k, n-1)$, and $n{-}k$ are pyramids over $\Delta(k-1,n-1)$.
\end{proposition}
\begin{proof}
	Lemma~\ref{eulerian:lemma:small_facets} shows the part about $n-k$ cells
	being pyramids over $\Delta(k-1,n-1)$.
	Lemma~\ref{eulerian:lemma:large_facets} gives us $k-1$ of the other cells.
	The last cell we need is just the cell from all points that got embedded at
	height one in $\tilde{\Delta}(k,n)$.
	
	The lattice volumes of these cells add up to the lattice volume of
	$\Delta(k,n)$ by the formula \ref{eulerian:recursion} and this finishes the
	proof.
\end{proof}

After studying the cells of the subdivision, we can proceed by showing that it
has no non-trivial coarsening. 

\begin{lemma}\label{eulerian:lemma:complete}
  For $2<k<n-2$ the dual graph of $\Delta(k,n)^\kappa$ is complete.
\end{lemma}
\begin{proof}
	In Lemma~\ref{eulerian:lemma:small_facets} and
	Lemma~\ref{eulerian:lemma:large_facets} we discussed exactly which equations
	cut out the base of the cells, all of which were pyramids over some
	hypersimplex. Since all pyramids had apex $c$, $c$ will also be contained in
	any intersection. For intersecting two cells, we just have to combine two
	equations. There are three cases:
	\begin{itemize}
		\item Assume $c_i=c_j=0$ and $x_i=x_j=1$. Taking $x_i=1$ gave us a contraction facet
		isomorphic to the hypersimplex $\Delta(n-1,k-1)$ and now selecting
		from these points those with $x_j=1$ results in a smaller facet which
		is isomorphic to $\Delta(n-2,k-2)$. This has dimension $n-3$, since
		we require $k>2$. Combined with $c$ the intersection has dimension
		$n-2$, hence it is a facet.
		\item Similarly assume $c_i=0$ and $c_j=1$. Thus we have the equations
		$x_i=1$ and $x_j=0$. This time the intersection consists of a base
		that is isomorphic to $\Delta(n-2,k-1)$ and apex $c$.
		\item For the last case we have $c_i=c_j=1$ and $x_i=x_j=1$. We arrive at
		$\Delta(n-2,k)$.
	\end{itemize}
\end{proof}


\begin{theorem}
  \label{thm:kappa}
  The regular subdivision $\Delta(k,n)^\kappa$ is a non-matroidal coarsest subdivision of spread $n$. 
  Out of the $n$ maximal cells, $k$ are pyramids of height one over $\Delta(k, n-1)$, and the remaining $n-k$ are pyramids over $\Delta(k-1,n-1)$.
\end{theorem}

\begin{proof}
  The fact that $\Delta(k,n)^\kappa$ is a coarsest subdivision follows from Corollary~\ref{cor:coarsest} and Lemma \ref{eulerian:lemma:complete} for $2<k<n-2$, and from Theorem \ref{thm:lambda} and Observation~\ref{obs:lambda=kappa} for $k \in \{ 2, n-2 \}$.
  Points on the Dressian are characterized by what they induce on the (octahedral) $3$-faces of $\Delta(k,n)$; these are the $3$-term Plücker relations.
  One half, i.e., $k+(n-k)=n$, of the facets of $\Delta(k,n)$ are not subdivided by $\kappa$.
  There are at least $n-k$ facets of type $\Delta(k, n-1)$ which are subdivided such that at least some of their octahedral $3$-faces are triangulated.
  This argument shows that $\kappa\not\in\Dr(k,n)$.
\end{proof}

The following argument shows that $\lambda_{k,n}$ and $\kappa_{k,n}$ are different, in general.
\begin{corollary}
  The regular subdivisions $\Delta(k,n)^\lambda$ and $\Delta(k,n)^\kappa$ are combinatorially isomorphic if and only if $k=2$.
\end{corollary}
\begin{proof}
  Observation~\ref{obs:lambda=kappa} deals with the case $k=2$.
  So let us assume that $k\geq 3$.
  Then an easy induction shows that $n-k+1$ is strictly smaller than $\tbinom{n-1}{k-1}=\tbinom{n}{k}-\tbinom{n-1}{k}$.
  Consequently, the large maximal cell $C$ of $\Delta(k,n)^\lambda$, which has $\tbinom{n}{k}-n+k$ vertices, is larger than any (maximal) cell of $\Delta(k,n)^\kappa$, which has at most $\tbinom{n-1}{k}+1$ vertices.
\end{proof}

\section{Finite Metric Spaces}
\label{sec:metric}

Here, we put a special focus on the second hypersimplices $\Delta(2,n)$, which are naturally related to finite metric spaces \cite{DezaLaurent:2010,T-theory} and phylogenetics \cite{SempleSteel:2003}; see also \cite[\S10.6]{ETC}.
We summarize the connection.
A \emph{dissimilarity map} on $n$ points is a map $D:\tbinom{[n]}{2}\to\RR_{\geq0}$.
Via the identification of the 2-element set $\{i, j\}\subset[n]$ with the point $e_{\{i,j\}}=e_i + e_j$, the dissimilarity map $D$ or its negative, $-D$, may be viewed as a lifting function on $\Delta(2,n)$.
In this case the lower envelope from \eqref{eq:envelope} becomes
\[
  \envelope_{-D}(\Delta(2,n)) \ = \ \bigl\{x\in\RR^{n}_{\geq 0} \bigmid x_i + x_j \geq D(i,j) \text{ for all } i \neq j \bigr\} \enspace .
\]
It is known that $-D$ lies in the Dressian $\Dr(2,n)$ if and only if the tight span of $\Delta(2,n)^{-D}$ is 1-dimensional, i.e., a tree \cite[Corollary 10.48]{ETC}.
Passing to the negative is natural here as the four-point condition is a max-tropical vanishing condition, which in turn asks for regular subdivisions defined via upper convex hulls; see \cite[Lemma~4.3.6]{Tropical+Book}, \cite[Theorem 10.47]{ETC} and Remark~\ref{rem:upper-lower}.

\subsection*{The metric cone}
Now we need to understand how metric spaces enter the picture.
A dissimilarity map $D:\tbinom{[n]}{2}\to\RR_{\geq0}$ is a \emph{pseudo-metric} if it satisfies the triangle inequality $D(i,k)\leq D(i,j)+D(j,k)$ for all $i,j,k\in[n]$.
Observe that the triangle inequalities form homogeneous linear inequalities in $\RR^{\tbinom{n}{2}}$.
So the set of all pseudo-metrics on $n$ points forms a cone which lies in the nonnegative orthant of $\RR^{\tbinom{n}{2}}$; this is called the \emph{metric cone} $\MC(n)$.
The latter cone is called the \enquote{semi-metric cone} in the monograph of Deza and Laurent \cite{DezaLaurent:2010}.
Table~\ref{tab:MC8:rays} shows the number of rays and their orbits of the metric cones $\MC(n)$ for $n\leq 8$; the values for $n\geq 9$ are unknown.
A pseudo-metric $D$ is called \enquote{primitive} in \cite[\S20.3]{DezaLaurent:2010} if it spans a ray of the metric cone.
Avis constructed infinite families of rays of $\MC(n)$ from graphs \cite{Avis:1980}.

\begin{table}[tb]
  \centering
  \caption{Number of rays  of the metric cone $\MC(n)$ and their orbits according to Deza and Dutour-Sikri\'c \cite[Table 1]{DezaDutour:2018}.
    The number of orbits of rays of $\MC(6)$ is erroneously reported to be seven in loc.\ cit.}
  \label{tab:MC8:rays}
  \begin{tabular*}{.8\linewidth}{@{\extracolsep{\fill}}lcccccc@{}}
    \toprule
    $n$ & 3 & 4 & 5 & 6 &  7 &    8 \\
    \midrule
     rays & 3 & 7 & 25 & 296 & 55{,}226 & 119{,}269{,}588 \\
     orbits & 1 & 2 &  3 &   8 &    46 &      3{,}918 \\
    \bottomrule
  \end{tabular*}
\end{table}

Any partition of $[n]$ with two parts, $A$ and $B$, gives rise to the \emph{split pseudo-metric}
\begin{equation}\label{eq:split-metric}
  D_{A,B}(i,j) \ = \
  \begin{cases}
    0 & \text{if } \{i,j\}\subseteq A \text{ or } \{i,j\}\subseteq B \\
    1 & \text{otherwise.}
  \end{cases}
\end{equation}
The map $D_{A,B}$ is called a \enquote{split semi-metric} in \cite{DezaDutour:2018}; the name \enquote{cut semi-metrics} is used in \cite{DezaLaurent:2010}.
With respect to the natural $\Sym(n)$-action there are only $\lfloor \tfrac{n}{2} \rfloor$ orbits, which we denote
\[
  D_{1,n-1}, D_{2,n-2}, \dots, D_{\lfloor n/2\rfloor, \lceil n/2\rceil} \enspace .
\]
The pseudo-metric $D_{A,B}$ lies in the orbit $D_{|A|,|B|}$ for $|A|\leq|B|$.
Observe that $D_{A,B} = -D_{B,A}+\1$.
That is, for the classification into types, the sign does not matter much here.

The split pseudo-metrics form rays of the metric cone.
Conversely, each ray of $\MC(n)$ corresponds to a split pseudo-metric for $n\leq 4$.
For $n\geq 5$ there are other rays, too.
\begin{example}
  Table~\ref{tab:MC8:rays} tells us that the metric cone $\MC(5)$ has 25 rays in three orbits.
  Only one of them does not correspond to a split pseudo-metric.
  This orbit is generated by $-\lambda_{2,5}$, where $\lambda_{2,5}$ is defined in \eqref{lambda}.
\end{example}

\subsection*{The metric fan}
Next we need to understand the interplay between $\MC(n)$ and $\Sigma(2,n)$.
The \emph{secondary metric cone} of the height function $\delta:\tbinom{[n]}{2}\to\RR$ is defined as
\begin{equation}
  \label{eq:secondary-metric-cone}
  \MC(\delta) \ := \ \seccone(\Delta(2,n)^{-\delta}) \cap \MC(n) \; ,
\end{equation}
this is the set of all pseudo-metrics on $n$ points which lie in the closed secondary cone of $\delta$.
Observe that each secondary metric cone is pointed, as it is contained in the nonnegative orthant.
Moreover, the secondary fan is complete, and so any height function on $\Delta(2,n)$ gives rise to a secondary metric cone.
A pseudo-metric is a \emph{metric} if additionally $D$ attains strictly positive values.
The following result translates between secondary cones and secondary metric cones.
\begin{lemma}\label{lem:metric}
  Let $\delta, \delta':\tbinom{[n]}{2}\to\RR$ be arbitrary.
  We have $\MC(\delta)=\MC(\delta')$ if and only if $\seccone(\Delta(2,n)^{-\delta})=\seccone(\Delta(2,n)^{-\delta'})$.
  That is, the map which takes the secondary cone $\seccone(\Delta(2,n)^{-\delta})$ to its secondary metric cone $\MC(\delta)$ is injective.
  Furthermore, we have $\dim \seccone(\Delta(2,n)^{-\delta})= \dim \MC(\delta)$.
\end{lemma}
\begin{proof}
  Let $\delta$ be an arbitrary height function on the vertices of $\Delta(2,n)$, which may attain negative values.
  Because the all-ones vector $\1$ lies in the lineality space of the secondary fan, the height function $\alpha\1-\delta$ induces the same subdivision on $\Delta(2,n)$ as $-\delta$, for any real number $\alpha$.
  When $\alpha$ is positive and large enough, the height function $\alpha\1-\delta$ is strictly positive and satisfies the triangle inequalities, and so $\alpha\1-\delta$ is a pseudo-metric.
  For the metric given by the vector $\1$ all triangle inequalities are strict; hence $\alpha\1-\delta$ even lies in the interior of the metric cone $\MC(n)$ for $\alpha\gg0$.

  By construction of the (closed) secondary cone $\seccone(\Delta(2,n)^{-\delta})$ the height function $\delta$ always lies in the relative interior of $\seccone(\Delta(2,n)^{-\delta})$.
  So we obtain $\dim\MC(\delta)=\dim\seccone(\Delta(2,n)^{\alpha\1-\delta})=\dim\seccone(\Delta(2,n)^{-\delta})$.
  The same argument also shows that $\MC(\delta)=\MC(\delta')$ implies that $\seccone(\Delta(2,n)^{-\delta})=\seccone(\Delta(2,n)^{-\delta'})$.
\end{proof}
Lemma \ref{lem:metric} is implicit in \cite{SturmfelsYu:2004}.
However, some details are subtle.
For instance, \cite[Corollary 10]{SturmfelsYu:2004} is not fully correct; see Theorem \ref{thm:Sigma26} and Corollary \ref{cor:MF6} below.
Next we describe how the split pseudo-metrics are related to the secondary fan.
\begin{remark}\label{rem:splits}
  The negatives of the pseudo-metrics of types $D_{2,n-2}, \dots, D_{\lfloor n/2\rfloor, \lceil n/2\rceil}$ induce splits of $\Delta(2,n)$.
  Up to symmetry, these are all the splits of $\Delta(2,n)$; see \cite[Theorem 5.3]{HerrmannJoswig:2008}.
  The subdivision of $\Delta(2,n)$ induced by the negative of a pseudo-metric of type $D_{1,n-1}$ is trivial; i.e., the vectors $D_{\{i\},[n]-\{i\}}$, for $i\in[n]$, lie in the lineality space of $\Sigma(2,n)$.
  Since these $n$ vectors are linearly independent, and the dimension of the lineality space equals $n$, the negatives of the pseudo-metrics of type $D_{1,n-1}$ span the entire lineality space.
\end{remark}

The \emph{metric fan} on $n$ points is the polyhedral fan $\MF(n)$ formed by the secondary metric cones.
The purpose of this section is to explore the relationship between the metric fan $\MF(n)$ and the secondary fan $\Sigma(2,n)$.
Our definition deviates from \cite{SturmfelsYu:2004}, where $\Sigma(2,n)$ and $\MF(n)$ are not distinguished.
\begin{proposition}\label{prop:metric}
  The metric fan $\MF(n)$ has the following types of rays:
  \begin{enumerate}
  \item negatives of rays of the secondary fan $\Sigma(2,n)$;
  \item one additional orbit, corresponding to the $n$ split pseudo-metrics of type $D_{1,n-1}$.
  \end{enumerate}
\end{proposition}

The rays of the metric fan are called \enquote{prime metrics} in \cite{KoolenMoultonTonges:2000}.

\begin{proof}
  The fan $\MF(n)$ is the intersection of the fan $\Sigma(2,n)$ with the cone $\MC(n)$.
  Each ray of $\MC(n)$ is also a ray of $\MF(n)$, because $\Sigma(2,n)$ is a complete fan.
  By Lemma~\ref{lem:metric} mapping a secondary cone to its secondary metric cone preserves the dimension.
  Consequently, the negative of any ray of $\Sigma(2,n)$ is also a ray of $\MF(n)$.
  
  In general, the intersection of a fan and a cone may have other rays, too.
  It remains to explain why this does not happen here.
  Each cone of the metric fan arises as the intersection of a secondary cone with the metric cone.
  The rays of such a secondary metric cone can be determined via the double description method, which is a classical (dual) convex hull algorithm.
  Here we need the homogeneous form; see, e.g., \cite[Algorithm 5.3]{Polyhedral+and+Algebraic+Methods} for a textbook reference.
  We let the double-description method start with the secondary cone, where we know both the facets and the minimal faces; this is the \emph{double description} from which the algorithm gets its name.
  Any ray of the secondary metric cone which is neither a ray of $\MC(n)$ nor a ray of $\Sigma(2,n)$ arises as the intersection of some $2$-dimensional secondary cone with a hyperplane from one of the triangle inequalities.
  Recall that the triangle inequalities define the facets of $\MC(n)$.
  The dimension of the lineality space of $\Sigma(2,n)$ equals $n > 2$, and each secondary cone contains that lineality space.
  Therefore, there is no secondary cone of a suitable dimension to produce any additional rays.

  Let $\delta$ be an arbitrary ray of $\MC(n)$.
  It remains to clarify that $-\delta$ is either also a ray of $\Sigma(2,n)$ or $\delta$ is of type $D_{n-1}$.
  Since $\delta$ is a ray of $\MF(n)$, the secondary metric cone $\MC(\delta)$ is one-dimensional.
  By Lemma~\ref{lem:metric} the secondary cone $\seccone(\Delta(2,n)^{-\delta})$ is at most one-dimensional modulo linealities.
  Either that dimension is exactly one, whence $-\delta$ is a ray of $\Sigma(2,n)$.
  Or $-\delta$ lies in the lineality space, which is generated by the pseudo-metrics of type $D_{1,n-1}$.
\end{proof}

\begin{table}[tb]
  \centering
  \caption{Spread statistics for coarsest subdivisions of $\Delta(2,6)$}
  \label{tab:Delta26:spreads}
  \begin{tabular*}{.8\linewidth}{@{\extracolsep{\fill}}lcccccc@{}}
    \toprule
    spread & 2 & 5 & 6 & 7 & 10 & 11 \\
    \midrule
    orbits & 2 & 2 & 2 & 3 & 1 & 3 \\
    \bottomrule
  \end{tabular*}
\end{table}

The next example illustrates Lemma~\ref{lem:metric} and Proposition~\ref{prop:metric}.
\begin{example}\label{exmp:thrackle}
  The \emph{thrackle metric}, $T=T_n$, on $n$ points is defined by $T(i,j) = T(j,i) = (j-i)\cdot (n-j+i)$ for $i<j$, see \cite{DeLoeraSturmfelsThomas:1995}.
  It is generic, i.e., the subdivision $\Delta(2,n)^{-T}$ is a triangulation, known as the \emph{thrackle triangulation} of the second hypersimplex.

  For instance, for $n=6$, we have $\dim\seccone(\Delta(2,6)^{-T})=15$, including the six-dimensional lineality space.
  That secondary cone is simplicial, whence $\seccone(\Delta(2,6)^{-T})$ has nine rays, modulo linealities.
  The secondary metric cone $\MC(T)$ is simplicial, too; it has 15 rays, as the secondary metric cones are pointed.
  Six rays are of type $D_{1,5}$, another six of type $D_{2,4}$, and finally, we have three rays of type $D_{3,3}$.
  In particular, all the rays of $\MC(T)$ are induced by split pseudo-metrics.
  That is to say, applying the Split Decomposition Theorem~\ref{thm:split-decomposition} to the thrackle metric $T_6$ yields a trivial split-prime part.
\end{example}

The regular triangulations of $\Delta(2,6)$ have been determined with \TOPCOM \cite{TOPCOM-paper}.
We verified his result with our own software \mptopcom \cite{mptopcom-paper}, and the output was further processed with \polymake \cite{DMV:polymake}.
Altogether we recover the following result; see Section~\ref{sec:computation} for further details on the computations.
\begin{theorem}[{Sturmfels and Yu~\cite{SturmfelsYu:2004}; using \cite{TOPCOM-paper}}]
  \label{thm:Sigma26}
  The hypersimplex $\Delta(2,6)$ has exactly 194{,}160 regular triangulations in $339$ orbits, with respect to the natural $\Sym(6)$ action.
  Moreover, it has exactly 13 orbits of regular coarsest subdivisions.
  The largest spread of a coarsest subdivision equals 11.
\end{theorem}
The spreads which occur for coarsest subdivisions of $\Delta(2,6)$ and how they are distributed over the 13 orbits are given in Table~\ref{tab:Delta26:spreads}.
Their tight spans are shown in Figure~\ref{fig:Delta26:tightspans}.
The (two orbits of) splits (whose tight span is a single edge) are omitted; and the two types of coarsest subdivisions of spread five share the same tight span; so ten distinct tight spans remain.
Using \polymake, Sturmfels and Yu processed the 339 orbits of maximal cones of $\MF(6)$ to find all rays.
Proposition~\ref{prop:metric} says the the fans $\Sigma(2,6)$ and $\MF(6)$ differ by one orbit of rays, of type $D_{1,5}$.
This explains the discrepancy between Theorem~\ref{thm:Sigma26} and \cite[Corollary 10]{SturmfelsYu:2004}, which erroneously claims that there are 14 orbits of regular coarsest subdivisions.
\begin{corollary}[{Koolen, Moulton and T\"onges \cite{KoolenMoultonTonges:2000}; Sturmfels and Yu~\cite{SturmfelsYu:2004}}]
  \label{cor:MF6}
  The metric fan $\MF(6)$ has exactly 14 orbits of rays, with respect to the natural $\Sym(6)$ action.
\end{corollary}


\subsection*{An example of a metric space on six points}
A motivation of our present work comes from the wish to analyze finite metric spaces beyond split decomposition; see Theorem~\ref{thm:split-decomposition}.
To illustrate our results we present one comprehensive example.
This is a finite metric space arising from short DNA fragments of six species of bees, where the distance between two taxa is given by the Hamming distance of the corresponding DNA fragments.
This is an example data set\footnote{\url{https://github.com/husonlab/splitstree6/blob/main/examples-testing/bees-with-wikipedia-images.stree6}, retrieved on January 3, 2024} 
for the software \texttt{SplitsTree6} \cite{splitstree}, which is explained in \cite{HusonBryant:2005}.
The same data was already used in \cite[Example 10.76]{ETC}.
One half of the (symmetric) distance matrix reads as follows:
\[
  \beta \ = \
  \begin{pmatrix}
    0.0 & 0.09010340 &	0.10339734 & 0.09601182	& 0.00443131 & 0.07533235\\
        & 0.0& 0.09305761 & 0.09010340 & 0.09305761 & 0.10044313\\
        & & 0.0& 0.11669128 & 0.10635155 & 0.10339734\\
        & & & 0.0 & 0.09896603	& 0.09896603\\
        & & & & 0.0 & 0.07828656\\
        & & & & & 0.0		
  \end{pmatrix}
\]
We read the matrix $-\beta$ as a height function on the vertices of the hypersimplex $\Delta(2,6)$; i.e., $-\beta$ lifts the vertex $e_{i}+e_{j}$ to the negative distance between $i$ and $j$ in the metric space.
Whenever we write a distance function as a single row vector, then we apply the descending lexicographic ordering for the vertices of $\Delta(2,n)$ as in Remark~\ref{rem:ordering}.
This amounts to reading the entries in a distance matrix above the diagonal and row-wise from top to bottom.

The secondary cone $\seccone(\Delta(2,6)^{-\beta})$ turns out to have full dimension 15, including the 6-dimensional lineality space; this implies that $-\beta$ induces a triangulation.
That secondary cone has nine rays corresponding to the coarsest regular subdivisions refining the triangulation $\Delta(2,6)^{-\beta}$.
Those rays lie in six distinct orbits.
Five out of the nine rays are splits, where $s_1,s_2,s_3,s_4$ lie in one orbit, and $s_5$ is in an orbit by itself; see Table~\ref{tab:beta:rays}.
In this section we write all vectors and matrices as pseudo-metrics, i.e., with nonnegative coefficients.
That is why Table~\ref{tab:beta:rays} lists negative rays.


\begin{table}[bt]
   \centering
   \caption{Negatives of the rays of the secondary cone of $\Delta(2,6)^{-\beta}$.
     The six orbits are separated by extra line space.
     Orbits sorted by descending coherency index}
   \label{tab:beta:rays}
   \small
   \begin{tabular*}{.9\linewidth}{@{\extracolsep{\fill}}clcl@{}}
     \toprule
     ray & coordinates & spread & coherency index \\
     \midrule
     $s_1$ & $(1, 1, 1, 0, 1, 0, 0, 1, 0, 0, 1, 0, 1, 0, 1)$ & 2 & $ 0.03175776 $\\ 
     $s_2$ & $(1, 1, 0, 0, 0, 0, 1, 1, 1, 1, 1, 1, 0, 0, 0)$ & 2 & $ 0.00886262 $\\ 
     $s_3$ & $(1, 0, 1, 0, 0, 1, 0, 1, 1, 1, 0, 0, 1, 1, 0)$ & 2 & $ 0.00664697 $\\ 
     $s_4$ & $(0, 1, 0, 0, 1, 1, 0, 0, 1, 1, 1, 0, 0, 1, 1)$ & 2 & $ 0.00147710 $\\ 
     \addlinespace
     $s_5$ & $(1, 1, 1, 0, 0, 0, 0, 1, 1, 0, 1, 1, 1, 1, 0)$ & 2 & $ 0.00516987 $\\ 
     \addlinespace
     $r_1$ & $(1, 2, 2, 0, 1, 1, 1, 1, 2, 2, 2, 1, 2, 1, 1)$ & 5 & $ 0.00369276 $\\ 
     \addlinespace
     $r_2$ & $(1, 2, 2, 1, 2, 1, 1, 2, 3, 2, 3, 2, 1, 2, 1)$ & 7 & $ 2\cdot 10^{-9}$\\ 
     \addlinespace
     $r_3$ & $(1, 2, 2, 1, 1, 1, 1, 2, 2, 2, 3, 1, 1, 1, 2)$ & 6 & $ 2.5\cdot 10^{-9} $\\ 
     \addlinespace
     $r_4$ & $(1, 2, 2, 1, 1, 1, 1, 2, 2, 2, 1, 1, 1, 1, 0)$ & 5 & $ 5\cdot 10^{-10} $\\ 
     \bottomrule
   \end{tabular*}
\end{table}

Now we can split decompose $-\beta$, as described in Theorem \ref{thm:split-decomposition}.
This amounts to subtracting from $-\beta$ the contributions (in terms of their coherency indices) from the split rays of $\seccone(\Delta(2,6)^{-\beta})$.
We obtain the split prime part $\beta_0$, which again we write as a distance matrix:
\[
  \beta_0 = 
  \begin{pmatrix}
     0.0 & 0.0376662 & 0.0561300 & 0.0524372 & 0.0044313 & 0.0420975 & \\
          & 0.0 & 0.0849335 & 0.0812408 & 0.0406204 & 0.0782866 & \\
          &  & 0.0 & 0.0997046 & 0.0590842 & 0.0893648 & \\
          &  &  & 0.0 & 0.0553914 & 0.0856721 & \\
          &  &  &  & 0.0 & 0.0450517 & \\
          &  &  &  &  & 0.0 & \\
  \end{pmatrix}
\]
Inspecting $\beta_0$ reveals that some coefficients are still fairly large, relative to the sizes of the coefficients of $\beta$.
That is to say, in this case, the split decomposition leaves a lot unaccounted for.
A second look at Table~\ref{tab:beta:rays} tells us that the non-split ray $r_1$, of spread five, is responsible for the bulk of the split-prime part.
The coherency indices of the remaining three rays, of spreads five, six and seven, are at least six orders of magnitude smaller than the contribution of $r_1$.
This example illustrates the practical usefulness of finding and analyzing non-split rays of low spread of $\Delta(2,n)$.


\section{Further Computations}
\label{sec:computation}

For the results in this section, which  are new, we employed \TOPCOM \cite{TOPCOM-paper}, \mptopcom \cite{mptopcom-paper} and \polymake \cite{DMV:polymake}.
The data and our code for post processing are available online\footnote{See \url{https://polymake.org/triangulations/} and \url{https://github.com/dmg-lab/coarsest_subdivisions_of_hypersimplices}}.
Our first result of this section has been obtained with \mptopcom and \polymake.
It is the analog of Theorem~\ref{thm:Sigma26} for metric spaces on seven points.
\begin{theorem}\label{thm:computation:Delta27}
  The hypersimplex $\Delta(2,7)$ has exactly 153{,}209{,}697{,}210 regular triangulations in 30{,}485{,}496 orbits, with respect to the natural $\Sym(7)$ action.
  Moreover, it has exactly 13{,}147 orbits of regular coarsest subdivisions.
  The largest spread of a coarsest subdivision equals 35.
\end{theorem}
The secondary fan $\Sigma(2,7)$ is 21-dimensional with a 7-dimensional lineality space.
It is known that the hypersimplex $\Delta(k,n)$ admits nonregular triangulations if and only if $k \geq 2$, $n - k \geq 2$ and $n \geq 6$; see \cite[Theorem 6.3.21]{Triangulations}.
\TOPCOM computed that $\Delta(2,7)$ has a total of 189{,}355{,}661{,}460 not necessarily regular triangulations in 37{,}676{,}752 orbits.
The next result has again been obtained with \mptopcom and \polymake.
\begin{theorem}\label{thm:computation:Delta36}
  The hypersimplex $\Delta(3,6)$ has exactly 61{,}035{,}863{,}100 regular triangulations in 42{,}489{,}025 orbits, with respect to the natural $\Sym(6)\rtimes\ZZ/2$ action.
  Moreover, it has exactly 28{,}589 orbits of regular coarsest subdivisions.
  The largest spread of a coarsest subdivision equals 45.
\end{theorem}
It turns out that each number between $2$ and $45$, except $4$, arises as the spread of a coarsest subdivision of $\Delta(3,6)$.
The secondary fan $\Sigma(3,6)$ is 20-dimensional with a 6-dimensional lineality space.
\TOPCOM computed that $\Delta(3,6)$ has a total of 85{,}793{,}497{,}200 not necessarily regular triangulations in 59{,}708{,}427 orbits, with respect to the $\Sym(6)\rtimes \ZZ_2$ action.
The orbit count with respect to the smaller group $\Sym(6)$ increases to 119{,}186{,}888.
The Dressian $\Dr(3,6)$ is much smaller; this is a $4$-dimensional fan (modulo linealities) with 65 rays in three orbits and 1005 maximal cones in seven orbits \cite{SpeyerSturmfels04,BendleEtAl:2024}.

Comparing the results of \TOPCOM and \mptopcom shows that the percentage of nonregular triangulations of the hypersimplices $\Delta(2,7)$ and $\Delta(3,6)$ is about $20\%$ and $30\%$, respectively.
For higher parameters, the non-regular triangulations will become the majority (by far).

\subsection*{Algorithms}
The procedure that we applied does not exploit any special features of the respective point set, which we denote $A$.
In particular, our analysis of $\Delta(2,7)$ and $\Delta(3,6)$ uses the same steps.

The computation begins with finding all regular subdivisions of $A$ up to symmetry with respect to a given permutation group action on $A$.
In \mptopcom this is accomplished by a parallel reverse search through the flip-graph \cite{mptopcom-paper}.
The current version 1.1.2 of \TOPCOM first enumerates all triangulations by a purely combinatorial method and then filters for the regular ones by solving linear programs.
We used the full group of affine linear transformations in both cases.
The group $\Sym(7)$ acting on $\Delta(2,7)$ has order 5040, whereas the group $\Sym(6)\rtimes\ZZ_2$ acting on $\Delta(3,6)$ has 1440 elements.

The regular triangulations of $A$ bijectively correspond to the vertices of the secondary polytope of $A$; its normal fan is the secondary fan.
Vertex coordinates for the secondary polytope are given by GKZ vectors of triangulations.
So, conceptually, the task of finding the coarsest regular subdivisions could then be performed via a convex hull computation.
However, for all interesting point sets, including $\Delta(2,7)$ and $\Delta(3,6)$, these computations are way too large to be feasible in practice; see \cite{polymake:2017} for a computational survey.
Therefore, the computation needs to be organized differently.

The crucial observation is that the number of coarsest subdivisions of $A$ is usually much smaller than the number of triangulations.
For instance, Theorems \ref{thm:computation:Delta27} and \ref{thm:computation:Delta36} exhibit that, for these examples, the number of regular triangulations is at least two orders of magnitude larger than the number of regular coarsest subdivisions.
So, it is affordable to compute the rays of the secondary cone of the regular triangulation which represents one orbit of regular triangulations.
This is also a (dual) convex hull computation, but a much smaller one.
Moreover, the task is embarrassingly parallel: each secondary cone can be dealt with individually.
The rays/coarsest subdivisions are computed multiple times, and it needs to be checked which of them lie in the same orbits.
This is done by finding the lexicographically minimal representative in each orbit.
The \polymake code is the script \texttt{rays\_of\_sec\_cones.pl}.

The final step is to compute the spread statistics of the the coarsest subdivisions.
To this end we construct one regular subdivision per orbit and count the maximal cells.
This is again a convex hull computation.

All our \polymake scripts and data files can be found at \url{https://github.com/dmg-lab/coarsest_subdivisions_of_hypersimplices}.
The datasets of rays and histograms that we computed are stored in \polymake's file format.

\subsection*{Technical details}
The computations with \mptopcom were run on a cluster at TU Berlin, running AlmaLinux.
For $\Delta(3,6)$ we used 256 slots, while for $\Delta(2,7)$ we only used 128 slots; the term \enquote{slot} is used by the cluster scheduling mechanism.
Since \mptopcom is output sensitive, it was not necessary to allocate large amounts of RAM, and 4GB of RAM per slot sufficed easily, while also boosting the computation using \mptopcom's caching mechanism.
The computation of all regular triangulations took about 13 hours for $\Delta(3,6)$ and was done in three segments using \mptopcom's checkpointing mechanism.
The computation for $\Delta(2,7)$ took roughly 132 hours in 10 segments.
While the resulting numbers of triangulations and the size of the inputs are in the same range, the time elapsed for the computations differs considerably.
The reason is that the results for $\Delta(2,7)$ were obtained several years ago with \mptopcom 1.0, but remained unpublished until now.
For $\Delta(3,6)$ we used \mptopcom 1.3 which has an improved mechanism for checking regularity.

\section{Conclusion}
\label{sec:conclusion}

It is an intriguing question if Corollary~\ref{cor:spread} also holds for rays of the Dressian.
While $\Dr(k,n)$ is known to contain $\TGr(k,n)$, the Dressian has maximal cones of various dimensions in general.
We are not aware of an argument that would imply that each ray of $\Dr(k,n)$ is contained in some maximal cone of dimension at least $\dim\TGr(k,n)$.
\begin{question}\label{quest:spread}
  Let $\omega$ be a ray of the Dressian $\Dr(k,n)$.
  Is it true that the spread of $\Delta(k,n)^\omega$ does not exceed $\tbinom{n-2}{k-1}-(k-1)\cdot(n-k-1)+1$?
\end{question}

The metric cones, and various of their relatives occur in optimization and elsewhere.
Many natural combinatorial conjectures arise, others have been refuted; see, e.g., Deza and Indik \cite{DezaIndik:2007}.
In our scenario a particularly interesting combinatorial object is the \emph{(primal) graph} of a cone; its nodes are the rays, and two such nodes are defined to be adjacent if the corresponding rays span a face of dimension two.
The same definition also makes sense for a fan.
The \emph{diameter} of a graph is the maximum length of a shortest path between any two nodes.
\begin{question}
  What are the diameters of the graphs of the cones $MC(n)$ and the fans $\Sigma(2,n)$ as functions of $n$?
\end{question}
For $n\leq 8$ the diameter of $\MC(n)$ is known to be at most three.
However, it is unknown if the diameter of the graph of $\MC(n)$ is bounded by any constant.
Even less is known about the graphs of $\Sigma(k,n)$.
The diameter of the graph of $\Sigma(2,5)$, shown in Figure~\ref{fig:Dr25}, equals two.

\subsection*{Acknowledgment}
We are indebted to Bernd Sturmfels for inspiring discussions concerning hypersimplices and metric spaces and to Andrei Com\u{a}neci and Benjamin Lorenz for valuable comments, both mathematical and software related.
Further, we thank Mathieu Dutour-Skiri\'c for confirming that the number of orbits of rays of $\MC(6)$, indeed, equals eight.
This corrects a minor error in \cite[Table 1]{DezaDutour:2018}; see also Table~\ref{tab:MC8:rays}.
The title of this article borrows from the title of a book by David Eisenbud \cite{Eisenbud:1995}.

\printbibliography

@Article{HerrmannJoswig:2008,
    AUTHOR = {Herrmann, Sven and Joswig, Michael},
     TITLE = {Splitting polytopes},
   JOURNAL = {M\"{u}nster J. Math.},
  FJOURNAL = {M\"{u}nster Journal of Mathematics},
    VOLUME = {1},
      YEAR = {2008},
     PAGES = {109--141},
      ISSN = {1867-5778},
   MRCLASS = {52B12},
  MRNUMBER = {2502496},
  url = {https://www.uni-muenster.de/FB10/mjm/vol_1/mjm_vol_1_05.pdf},
  arxiv =	 {0805.0774},
  doi = {10.48550/arXiv.0805.0774}
}

@article {MR2739494,
    AUTHOR = {Herrmann, Sven},
     TITLE = {On the facets of the secondary polytope},
   JOURNAL = {J. Combin. Theory Ser. A},
  FJOURNAL = {Journal of Combinatorial Theory. Series A},
    VOLUME = {118},
      YEAR = {2011},
    NUMBER = {2},
     PAGES = {425--447},
      ISSN = {0097-3165,1096-0899},
   MRCLASS = {52B70 (05B35)},
  MRNUMBER = {2739494},
       DOI = {10.1016/j.jcta.2010.08.003},
       URL = {https://doi.org/10.1016/j.jcta.2010.08.003},
}

@article {mptopcom-paper,
    AUTHOR = {Jordan, Charles and Joswig, Michael and Kastner, Lars},
     TITLE = {Parallel enumeration of triangulations},
   JOURNAL = {Electron. J. Combin.},
  FJOURNAL = {Electronic Journal of Combinatorics},
    VOLUME = {25},
      YEAR = {2018},
    NUMBER = {3},
     PAGES = {Paper 3.6, 27},
      ISSN = {1077-8926},
   MRCLASS = {52B55 (52B70 68U05)},
  MRNUMBER = {3829292},
  url = {http://www.combinatorics.org/ojs/index.php/eljc/article/view/v25i3p6},
  arxiv = {1709.04746},
  DOI = {10.37236/7318}
}

@BOOK{Triangulations,
    AUTHOR = {De Loera, Jes{\'u}s A. and Rambau, J{\"o}rg and Santos, Francisco},
     TITLE = {Triangulations},
    SERIES = {Algorithms and Computation in Mathematics},
    VOLUME = {25},
      NOTE = {Structures for algorithms and applications},
 PUBLISHER = {Springer-Verlag},
   ADDRESS = {Berlin},
      YEAR = {2010},
     _PAGES = {xiv+535},
      ISBN = {978-3-642-12970-4},
   MRCLASS = {52B55 (05C10 52B05 57Q15 68U05)},
  MRNUMBER = {2743368 (2011j:52037)},
  DOI = {10.1007/978-3-642-12971-1}
}

@article {Speyer:2008,
    AUTHOR = {Speyer, David E.},
     TITLE = {Tropical linear spaces},
   JOURNAL = {SIAM J. Discrete Math.},
  FJOURNAL = {SIAM Journal on Discrete Mathematics},
    VOLUME = {22},
      YEAR = {2008},
    NUMBER = {4},
     PAGES = {1527--1558},
      ISSN = {0895-4801,1095-7146},
   MRCLASS = {52B40},
  MRNUMBER = {2448909},
       DOI = {10.1137/080716219},
       URL = {https://doi.org/10.1137/080716219},
}

@InProceedings{Foata:1977,
 Author = {Foata, Dominique},
 Title = {Distributions eul\'{e}riennes et {Mahoniennes} sur le groupe des permutations},
 Year = {1977},
 booktitle = {Higher combinatorics},
 editor = {Aigner, Martin},
 note = {Proceedings of the NATO Advanced Study Institute held in Berlin (West Germany), September 1--10, 1976.},
 pages = {27-49},
 zbMATH = {3698940},
 Zbl = {0447.05004},
}

@InProceedings{Stanley:1977,
  author = {Stanley, Richard P.},
  title = {Eulerian partitions of a unit hypercube},
  year = 1977,
  booktitle = {Higher combinatorics},
  editor = {Aigner, Martin},
  note = {Proceedings of the NATO Advanced Study Institute held in Berlin (West Germany), September 1--10, 1976.},
  pages = {49},
  zbl = {0359.05001},
  zbmath = 3559582,
}

@article {SturmfelsYu:2004,
    AUTHOR = {Sturmfels, Bernd and Yu, Josephine},
     TITLE = {Classification of six-point metrics},
   JOURNAL = {Electron. J. Combin.},               
  FJOURNAL = {Electronic Journal of Combinatorics},
    VOLUME = {11},                                 
      YEAR = {2004},
    NUMBER = {1},   
     PAGES = {Research Paper 44, 16},
      ISSN = {1077-8926},
   MRCLASS = {51K05 (05C12 52B05 52B55)},
  MRNUMBER = {2097310},
MRREVIEWER = {Ezra N. Miller},
       URL = {http://www.combinatorics.org/Volume_11/Abstracts/v11i1r44.html},
       DOI = {10.37236/1797}
}

@article {Hirai:2006a,
    AUTHOR = {Hirai, Hiroshi},
     TITLE = {A geometric study of the split decomposition},
   JOURNAL = {Discrete Comput. Geom.},
  FJOURNAL = {Discrete \& Computational Geometry. An International Journal of Mathematics and Computer Science},
    VOLUME = {36},
      YEAR = {2006},
    NUMBER = {2},
     PAGES = {331--361},
      ISSN = {0179-5376},
   MRCLASS = {52A41 (52C35)},
  MRNUMBER = {2252108},
MRREVIEWER = {Yaokun Wu},
       DOI = {10.1007/s00454-006-1243-1},
       URL = {https://doi.org/10.1007/s00454-006-1243-1},
}

@article {BandeltDress:1992,
    AUTHOR = {Bandelt, Hans-J\"{u}rgen and Dress, Andreas W. M.},
     TITLE = {A canonical decomposition theory for metrics on a finite set},
   JOURNAL = {Adv. Math.},
  FJOURNAL = {Advances in Mathematics},
    VOLUME = {92},
      YEAR = {1992},
    NUMBER = {1},
     PAGES = {47--105},
      ISSN = {0001-8708,1090-2082},
   MRCLASS = {54E35 (05C99)},
  MRNUMBER = {1153934},
MRREVIEWER = {Henry\ Martyn\ Mulder},
       DOI = {10.1016/0001-8708(92)90061-O},
       URL = {https://doi.org/10.1016/0001-8708(92)90061-O},
}

@book{ETC,
    AUTHOR = {Joswig, Michael},
     TITLE = {Essentials of tropical combinatorics},
    SERIES = {Graduate Studies in Mathematics},
    VOLUME = {219},
 PUBLISHER = {American Mathematical Society, Providence, RI},
      YEAR = {2021},
     _PAGES = {xx+398},
      ISBN = {978-1-4704-6653-4},
   MRCLASS = {14T15 (05E14 14T90 52Bxx 90C24)},
  MRNUMBER = {4423372},
  DOI = {10.1090/gsm/219}
}

@InProceedings{TOPCOM-paper,
  author = 	 {Rambau, Jörg},
  title = 	 {{\TOPCOM}: Triangulations of Point Configurations and Oriented Matroids},
  booktitle = {Mathematical Software --- ICMS 2002},
  year = 	 2002,
  editor = 	 {Cohen, A.M. and Gao, X.-S. and Takayama, N.},
  pages = 	 {330-340},
  publisher = {World Scientific},
  DOI = {10.1142/9789812777171_0035}
}

@incollection {DMV:polymake,
    AUTHOR = {Gawrilow, Ewgenij and Joswig, Michael},
     TITLE = {polymake: a framework for analyzing convex polytopes},
 BOOKTITLE = {Polytopes---combinatorics and computation (Oberwolfach, 1997)},
    SERIES = {DMV Sem.},
    VOLUME = {29},
     PAGES = {43--73},
 PUBLISHER = {Birkh\"auser},
   ADDRESS = {Basel},
      YEAR = {2000},
   MRCLASS = {52B55 (68U05)},
  MRNUMBER = {MR1785292 (2001f:52033)},
  DOI = {10.1007/978-3-0348-8438-9_2}
}

@book {SempleSteel:2003,
    AUTHOR = {Semple, Charles and Steel, Mike},
     TITLE = {Phylogenetics},
    SERIES = {Oxford Lecture Series in Mathematics and its Applications},
    VOLUME = {24},
 PUBLISHER = {Oxford University Press, Oxford},
      YEAR = {2003},
     _PAGES = {xiv+239},
      ISBN = {0-19-850942-1},
   MRCLASS = {92D15 (05C05 05C90 92D40)},
  MRNUMBER = {2060009},
MRREVIEWER = {Vincent L. Moulton},
DOI = {10.1093/oso/9780198509424.001.0001}
}

@article {T-theory,
    AUTHOR = {Dress, Andreas and Moulton, Vincent and Terhalle, Werner},
     TITLE = {{$T$}-theory: an overview},
      NOTE = {Discrete metric spaces (Bielefeld, 1994)},
   JOURNAL = {European J. Combin.},
  FJOURNAL = {European Journal of Combinatorics},
    VOLUME = {17},
      YEAR = {1996},
    NUMBER = {2-3},
     PAGES = {161--175},
      ISSN = {0195-6698},
   MRCLASS = {05C05},
  MRNUMBER = {1379369},
MRREVIEWER = {Lane Clark},
       DOI = {10.1006/eujc.1996.0015},
       URL = {https://doi.org/10.1006/eujc.1996.0015},
}

@Misc{splitstree,
  author = {Huson, Daniel H. and Bryant, David},
  howpublished = {\url{https://github.com/husonlab/splitstree6}}}

@article{HusonBryant:2005,
    author = {Huson, Daniel H. and Bryant, David},
    title = {Application of Phylogenetic Networks in Evolutionary Studies},
    journal = {Molecular Biology and Evolution},
    volume = {23},
    number = {2},
    pages = {254-267},
    year = {2005},
    month = {10},
    abstract = "{The evolutionary history of a set of taxa is usually represented by a phylogenetic tree, and this model has greatly facilitated the discussion and testing of hypotheses. However, it is well known that more complex evolutionary scenarios are poorly described by such models. Further, even when evolution proceeds in a tree-like manner, analysis of the data may not be best served by using methods that enforce a tree structure but rather by a richer visualization of the data to evaluate its properties, at least as an essential first step. Thus, phylogenetic networks should be employed when reticulate events such as hybridization, horizontal gene transfer, recombination, or gene duplication and loss are believed to be involved, and, even in the absence of such events, phylogenetic networks have a useful role to play. This article reviews the terminology used for phylogenetic networks and covers both split networks and reticulate networks, how they are defined, and how they can be interpreted. Additionally, the article outlines the beginnings of a comprehensive statistical framework for applying split network methods. We show how split networks can represent confidence sets of trees and introduce a conservative statistical test for whether the conflicting signal in a network is treelike. Finally, this article describes a new program, SplitsTree4, an interactive and comprehensive tool for inferring different types of phylogenetic networks from sequences, distances, and trees.}",
    issn = {0737-4038},
    doi = {10.1093/molbev/msj030},
    url = {https://doi.org/10.1093/molbev/msj030},
    _eprint = {https://academic.oup.com/mbe/article-pdf/23/2/254/3894375/msj030.pdf},
}

@incollection {KoolenMoultonTonges:1998,
    AUTHOR = {Koolen, Jack and Moulton, Vincent and T\"{o}nges, Udo},
     TITLE = {The coherency index},
      NOTE = {Discrete metric spaces (Villeurbanne, 1996)},
   JOURNAL = {Discrete Math.},
  FJOURNAL = {Discrete Mathematics},
    VOLUME = {192},
      YEAR = {1998},
    NUMBER = {1-3},
     PAGES = {205--222},
      ISSN = {0012-365X,1872-681X},
   MRCLASS = {54E35},
  MRNUMBER = {1656733},
       DOI = {10.1016/S0012-365X(98)00072-7},
       URL = {https://doi.org/10.1016/S0012-365X(98)00072-7},
}

@incollection {KoolenMoultonTonges:2000,
    AUTHOR = {Koolen, Jack and Moulton, Vincent and T\"{o}nges, Udo},
     TITLE = {A classification of the six-point prime metrics},
      NOTE = {Discrete metric spaces (Marseille, 1998)},
   JOURNAL = {European J. Combin.},
  FJOURNAL = {European Journal of Combinatorics},
    VOLUME = {21},
      YEAR = {2000},
    NUMBER = {6},
     PAGES = {815--829},
      ISSN = {0195-6698,1095-9971},
   MRCLASS = {54E35 (05C12)},
  MRNUMBER = {1791209},
       DOI = {10.1006/eujc.2000.0386},
       URL = {https://doi.org/10.1006/eujc.2000.0386},
}

@article {Schroeter:2019,
    AUTHOR = {Schr\"{o}ter, Benjamin},
     TITLE = {Multi-splits and tropical linear spaces from nested matroids},
   JOURNAL = {Discrete Comput. Geom.},
  FJOURNAL = {Discrete \& Computational Geometry. An International Journal of Mathematics and Computer Science},
    VOLUME = {61},
      YEAR = {2019},
    NUMBER = {3},
     PAGES = {661--685},
      ISSN = {0179-5376,1432-0444},
   MRCLASS = {52B40 (05A18 14T05)},
  MRNUMBER = {3918552},
MRREVIEWER = {Josephine\ T.\ Yu},
       DOI = {10.1007/s00454-018-0021-1},
       URL = {https://doi.org/10.1007/s00454-018-0021-1},
}

@ARTICLE{SpeyerSturmfels04,
  author = {Speyer, David and Sturmfels, Bernd},
  title = {The tropical {G}rassmannian},
  journal = {Adv. Geom.},
  year = {2004},
  volume = {4},
  pages = {389--411},
  number = {3},
  coden = {AGDEA3},
  fjournal = {Advances in Geometry},
  issn = {1615-715X},
  mrclass = {14P99 (13J30 14M25)},
  mrnumber = {MR2071813 (2005d:14089)},
  mrreviewer = {Meirav Amram-Blei},
  DOI = {10.1515/advg.2004.023}
}

@book {Tropical+Book,
    AUTHOR = {Maclagan, Diane and Sturmfels, Bernd},
     TITLE = {Introduction to tropical geometry},
    SERIES = {Graduate Studies in Mathematics},
    VOLUME = {161},
 PUBLISHER = {American Mathematical Society, Providence, RI},
      YEAR = {2015},
     _PAGES = {xii+363},
      ISBN = {978-0-8218-5198-2},
   MRCLASS = {14T05 (05B35 15A80 52B70)},
  MRNUMBER = {3287221},
  DOI = {10.1090/gsm/161}
}

@book {DezaLaurent:2010,
    AUTHOR = {Deza, Michel and Laurent, Monique},
     TITLE = {Geometry of cuts and metrics},
    SERIES = {Algorithms and Combinatorics},
    VOLUME = 15,
 PUBLISHER = {Springer-Verlag, Berlin},
      YEAR = 1997,
     _PAGES = {xii+587},
      ISBN = {3-540-61611-X},
   MRCLASS = {52-02 (05B30 05C12 51K05 52B12 52C07 68R05 90C28)},
  MRNUMBER = 1460488,
MRREVIEWER = {Alexander\ I.\ Barvinok},
       DOI = {10.1007/978-3-642-04295-9},
       URL = {https://doi.org/10.1007/978-3-642-04295-9},
}

@article {DeLoeraSturmfelsThomas:1995,
    AUTHOR = {De Loera, Jes\'{u}s A. and Sturmfels, Bernd and Thomas, Rekha R.},
     TITLE = {Gr\"{o}bner bases and triangulations of the second hypersimplex},
   JOURNAL = {Combinatorica},
  FJOURNAL = {Combinatorica. An International Journal on Combinatorics and the Theory of Computing},
    VOLUME = {15},
      YEAR = {1995},
    NUMBER = {3},
     PAGES = {409--424},
      ISSN = {0209-9683},
   MRCLASS = {13P10 (52B20 52B40)},
  MRNUMBER = {1357285},
MRREVIEWER = {Alexander\ I.\ Barvinok},
       DOI = {10.1007/BF01299745},
       URL = {https://doi.org/10.1007/BF01299745},
}

@article {Avis:1980,
    AUTHOR = {Avis, David},
     TITLE = {On the extreme rays of the metric cone},
   JOURNAL = {Canadian J. Math.},
  FJOURNAL = {Canadian Journal of Mathematics. Journal Canadien de Math\'{e}matiques},
    VOLUME = {32},
      YEAR = {1980},
    NUMBER = {1},
     PAGES = {126--144},
      ISSN = {0008-414X,1496-4279},
   MRCLASS = {52A25 (68E10 90C15)},
  MRNUMBER = {559790},
MRREVIEWER = {J.\ Parida},
       DOI = {10.4153/CJM-1980-010-0},
       URL = {https://doi.org/10.4153/CJM-1980-010-0},
}

@article {DezaDutour:2018,
    AUTHOR = {Deza, Michel and Dutour Sikiri\'{c}, Mathieu},
     TITLE = {The hypermetric cone and polytope on eight vertices and some generalizations},
   JOURNAL = {J. Symbolic Comput.},
  FJOURNAL = {Journal of Symbolic Computation},
    VOLUME = {88},
      YEAR = {2018},
     PAGES = {67--84},
      ISSN = {0747-7171,1095-855X},
   MRCLASS = {52B12 (51K05)},
  MRNUMBER = {3777382},
MRREVIEWER = {Zsolt\ L\'{a}ngi},
       DOI = {10.1016/j.jsc.2016.01.009},
       URL = {https://doi.org/10.1016/j.jsc.2016.01.009},
}

@article{polymake:2017,
    author = {Assarf, Benjamin and Gawrilow, Ewgenij and Herr, Katrin and Joswig, Michael and Lorenz, Benjamin and Paffenholz, Andreas and Rehn, Thomas},
     title = {Computing convex hulls and counting integer points with \polymake},
   journal = {Math. Program. Comput.},
    volume = {9},
      year = {2017},
    number = {1},
     pages = {1-38},
   mrclass = {90C57 (52 90-04)},
  mrnumber = {3613012},
       doi = {10.1007/s12532-016-0104-z},
       url = {http://dx.doi.org/10.1007/s12532-016-0104-z},
     arxiv = {1408.4653v2},
}

@Book{Polyhedral+and+Algebraic+Methods,
    AUTHOR = {Joswig, Michael and Theobald, Thorsten},
     TITLE = {Polyhedral and algebraic methods in computational geometry},
    SERIES = {Universitext},
      NOTE = {Revised and updated translation of the 2008 German original},
 PUBLISHER = {Springer},
   ADDRESS = {London},
      YEAR = {2013},
     _PAGES = {x+250},
      ISBN = {978-1-4471-4816-6},
   MRCLASS = {51-01 (13P10 52-01 52B55)},
  MRNUMBER = {2905853},
       DOI = {10.1007/978-1-4471-4817-3},
       URL = {http://dx.doi.org/10.1007/978-1-4471-4817-3},
}

@book {Eisenbud:1995,
    AUTHOR = {Eisenbud, David},
     TITLE = {Commutative algebra},
    SERIES = {Graduate Texts in Mathematics},
    VOLUME = {150},
      NOTE = {With a view toward algebraic geometry},
 PUBLISHER = {Springer-Verlag, New York},
      YEAR = {1995},
     _PAGES = {xvi+785},
      ISBN = {0-387-94268-8; 0-387-94269-6},
   MRCLASS = {13-01 (14A05)},
  MRNUMBER = {1322960},
MRREVIEWER = {Matthew\ Miller},
       DOI = {10.1007/978-1-4612-5350-1},
       URL = {https://doi.org/10.1007/978-1-4612-5350-1},
}

@article {DezaIndik:2007,
    AUTHOR = {Deza, Antoine and Indik, Gabriel},
     TITLE = {A counterexample to the dominating set conjecture},
   JOURNAL = {Optim. Lett.},
  FJOURNAL = {Optimization Letters},
    VOLUME = {1},
      YEAR = {2007},
    NUMBER = {2},
     PAGES = {163--169},
      ISSN = {1862-4472,1862-4480},
   MRCLASS = {05C69 (90C10 90C57)},
  MRNUMBER = {2357597},
MRREVIEWER = {Guntram\ Scheithauer},
       DOI = {10.1007/s11590-006-0001-x},
       URL = {https://doi.org/10.1007/s11590-006-0001-x},
}

@article {BendleEtAl:2024,
    AUTHOR = {Bendle, Dominik and B\"{o}hm, Janko and Ren, Yue and Schr\"{o}ter, Benjamin},
     TITLE = {Massively parallel computation of tropical varieties, their positive part, and tropical {G}rassmannians},
   JOURNAL = {J. Symbolic Comput.},
  FJOURNAL = {Journal of Symbolic Computation},
    VOLUME = {120},
      YEAR = {2024},
     PAGES = {Paper No. 102224, 28},
      ISSN = {0747-7171,1095-855X},
   MRCLASS = {14T15 (14M15 14Q15 52B15 68W10 68W30)},
  MRNUMBER = {4583112},
       DOI = {10.1016/j.jsc.2023.102224},
       URL = {https://doi.org/10.1016/j.jsc.2023.102224},
}

@misc{mardi,
  author       = {The MaRDI consortium},
  title        = {{MaRDI: Mathematical Research Data Initiative 
                   Proposal}},
  month        = may,
  year         = 2022,
  publisher    = {Zenodo},
  doi          = {10.5281/zenodo.6552436},
  url          = {https://doi.org/10.5281/zenodo.6552436}
}

\end{document}